\documentclass[a4wide,10pt]{amsart}
\usepackage{amsmath,amsthm, amssymb, color}
\usepackage{tikz,subcaption}
\usetikzlibrary{arrows,snakes,backgrounds}
\usepackage{amsfonts, amscd}
\usepackage{epsfig,multicol}
\usepackage{comment}
\usepackage{url}
\usepackage{multirow,bigdelim}

\theoremstyle{plain}
\newtheorem{theo}{Theorem}[section]
\newtheorem{lemm}[theo]{Lemma}
\newtheorem{coro}[theo]{Corollary}
\newtheorem{prop}[theo]{Proposition}

\theoremstyle{definition}
\newtheorem{defi}{Definition}[section]

\theoremstyle{remark}
\newtheorem{rema}{Remark}
\newtheorem*{rema*}{Remark}

\makeatletter

\@addtoreset{equation}{section}
\makeatother
\newcommand{\C}{\mathbb{C}}
\newcommand{\CP}{\mathbb{P}}
\newcommand{\Q}{\mathbb{Q}}
\newcommand{\R}{\mathbb{R}}
\newcommand{\Z}{\mathbb{Z}}

\newcommand{\ds}{\displaystyle}
\newcommand{\uC}{\underline{\mathbb{C}}}

\newcommand{\GL}{\text{GL}}

\def\a{\mathbf a}
\def\b{\mathbf b}
\def\c{\mathbf c}

\def\e{\mathbf e}
\def\g{\boldsymbol g}
\def\h{\boldsymbol h}

\def\r{\mathbf r}
\def\u{\mathbf u}
\def\v{\mathbf v}
\def\x{\mathbf x}
\def\z{\boldsymbol z}

\def\0{\mathbf 0}
\def\1{\mathbf 1}

\def\w{\mathbf w}
\def\n{\mathbf n}

\def\ZP{\mathcal{Z}_P}
\def\cB{\mathcal{B}}
\def\cK{\mathcal{K}}

\def\cV{\mathcal{V}}
\def\cK{\mathcal{K}}
\def\cI{\mathcal{I}}

\def\N_R{N_{\R}}
\def\M_R{M_{\R}}
\def\s<#1>{\langle #1 \rangle}
\def\l<#1>{\left\langle #1 \right\rangle}
\def\VAR{\mathcal{VAR}}
\def\cC{\mathcal{C}}

\DeclareMathOperator{\rank}{rank}
\DeclareMathOperator{\Ann}{Ann}
\DeclareMathOperator{\diag}{diag}
\DeclareMathOperator{\PU}{PU}
\DeclareMathOperator{\SU}{SU}
\DeclareMathOperator{\U}{U}
\DeclareMathOperator{\Int}{Int}

\DeclareMathOperator{\cone}{cone}
\DeclareMathOperator{\Aut}{Aut}
\DeclareMathOperator{\vt}{vc}
\DeclareMathOperator{\vcd}{\cC_n}

\makeatletter

\@addtoreset{equation}{section}
\makeatother

\begin{document}
\title[Toric manifolds over an $n$-cube with one vertex cut]{Classification of toric manifolds over an $n$-cube with one vertex cut}

\author[S. Hasui]{Sho Hasui}
\address{{Faculty of Liberal Arts and Sciences, Osaka Prefecture University, Osaka 599-8531, Japan.}}
\email{{s.hasui@las.osakafu-u.ac.jp}}

\author[H. Kuwata]{Hideya Kuwata}
\address{Department of Mathematics, Osaka City University, Sumiyoshi-ku, Osaka 558-8585, Japan.}
\email{hideya0813@gmail.com}

\author[M. Masuda]{Mikiya Masuda}
\address{Department of Mathematics, Osaka City University, Sumiyoshi-ku, Osaka 558-8585, Japan.}
\email{masuda@sci.osaka-cu.ac.jp}

\author[S. Park]{Seonjeong Park}
\address{Department of Mathematics, Osaka City University, Sumiyoshi-ku, Osaka 558-8585, Japan.}
\email{seonjeong1124@gmail.com}

\date{\today}

\thanks{The third author was partially supported by JSPS Grant-in-Aid for Scientific Research 16K05152}
\subjclass[2000]{Primary 55N10, 57S15; Secondary 14M25}
\keywords{toric manifold, polytope, Oda's $3$-fold, moment-angle manifold, cohomological rigidity.}

\begin{abstract}
We say that a {complete nonsingular} toric variety (called a toric manifold in this paper) {is over $P$ if its quotient} by the compact torus is homeomorphic to $P$ as a manifold with corners. Bott manifolds (or Bott towers) are toric manifolds over an $n$-cube $I^n$ and blowing them up at a fixed point produces toric manifolds over $\vt(I^n)$ an $n$-cube with one vertex cut. They are all projective. On the other hand, Oda's {$3$-fold, the} simplest non-projective toric manifold, is over $\vt(I^3)$. In this paper, we classify toric manifolds over $\vt(I^n)$ $(n\ge 3)$ as varieties and also as smooth manifolds. As a consequence, it turns out that (1) there are many non-projective  {toric manifolds} over $\vt(I^n)$ {but they} are all diffeomorphic, and (2) toric manifolds over $\vt(I^n)$  {in some class are determined by their cohomology rings} as varieties  {among toric manifolds}.
\end{abstract}
\maketitle
\bigskip
\section{Introduction}
A \textit{toric variety} of complex dimension $n$ is a normal complex algebraic variety with an algebraic action of $(\C^*)^n$ having an open dense orbit. In this paper, we are concerned with complete nonsingular toric varieties and call them \textit{toric manifolds}. As is well-known, the category of toric varieties is equivalent to the category of fans, by which the classification of toric manifolds as varieties reduces to a problem of combinatorics. Indeed, this fundamental fact enables us to classify toric manifolds of complex dimension $\le 2$ as varieties and also as smooth manifolds (see \cite{fult93}, \cite{oda88}).

However, the classification of fans up to isomorphism is not an easy task in general, and not much is known {about} the classification of toric manifolds as smooth (or topological) manifolds in complex dimension {$\ge 3$}. An intriguing problem posed in \cite{ma-su07} for the classification of toric manifolds as smooth (or topological) manifolds is what is now called the \textit{cohomological rigidity problem}, which asks whether toric manifolds are diffeomorphic (or homeomorphic) if their integral cohomology rings are isomorphic as graded rings.

Let $X$ be a toric manifold of complex dimension $n$. Its quotient by the compact torus $(S^1)^n$ of $(\C^*)^n$ is an $n$-dimensional manifold with corners.
It is often homeomorphic to a simple polytope as a manifold with corners, e.g., this is the case when $X$ is projective or $n\le 3$ but not the case in general (\cite{suya14}). When the quotient is homeomorphic to $P$ as a manifold {with} corners, we say that $X$ is over $P$.

Toric manifolds over an $n$-cube $I^n$ are called Bott manifolds (or Bott towers) and form an interesting {class} of toric manifolds (\cite{gr-ka94}). A Bott tower is a sequence of an iterated $\CP^1$-bundles starting with a point, where each $\CP^1$-bundle is {the} projectivization of the Whitney sum of two complex line bundles, and a Bott manifold is the top manifold in the sequence. The cohomological rigidity problem is not solved even for Bott manifolds but many {results have been produced in support of the affirmative answer to the problem, see~\cite{choi15}, \cite{C-M12}, \cite{C-M-M15}, \cite{ch-ma-su10}, \cite{ma-pa08}.}

Let $\vt(I^n)$ be an $n$-cube with one vertex cut. Blowing up Bott manifolds at a fixed point produces toric manifolds over $\vt(I^n)$. They are all projective since so are Bott manifolds. On the other hand, Oda's $3$-fold, which is known as the simplest non-projective toric manifold, is over $\vt(I^3)$. {So, it would be meaningful to classify toric manifolds over $\vt(I^n)$ as varieties and also as smooth manifolds, and we carry out the task in this paper.}

In order to state our {main results}, we introduce some terminology. We assume $n\ge 3$. The boundary complex $\vcd$ of the {simplicial polytope dual to} $\vt(I^n)$ is the underlying simplicial complex of the fans associated with the toric manifolds over $\vt(I^n)$. The simplicial complex $\vcd$ has $2n+1$ vertices, which we {give labels $1,2,\dots,2n+1$.} The vertex corresponding to the facet of $\vt(I^n)$ obtained by cutting a vertex of an $n$-cube $I^n$ is a distinguished vertex and we label it $2n+1$. There is a unique $(n-1)$-simplex in $\vcd$ which does not {intersect with the link of} the vertex $2n+1$. {The $n$ vertices of the simplex are labelled by} $1,2,\dots,n$. Then the labels of the remaining vertices of $\vcd$ are uniquely determined by requiring that the vertices $i$ and $n+i$ do not span a $1$-simplex of $\vcd$ for each $i=1,\dots,n$.

Let $X$ be a toric manifold over $\vt(I^n)$. We denote by $\v_i$ the primitive edge vector in the fan of $X$ corresponding to the vertex $i$ of $\vcd$. Then $\{\v_1,\v_2,\dots,\v_n\}$ forms a basis of the lattice of the fan {so that} we obtain an integer square matrix $A_X$ {satisfying}
\[
(\v_{n+1}, \v_{n+{2}},\dots,\v_{2n})=-(\v_1,\v_2,\dots,\v_n)A_X.
\]
It is not difficult to see that $\det A_X=1$ if and only if $X$ is the blow-up of a Bott manifold at a fixed point, {where the Bott manifold is associated with the fan obtained from the fan of $X$ by removing the vertex $2n+1$.}   {The following is our first main theorem, which} follows from Propositions~\ref{prop: cohomology ring of Type 3},~\ref{prop:coho_iso_det},~\ref{prop:diffeomorphism class of Type 0}, and~\ref{prop: diffeomorphism class of Type 2}.

\begin{theo} \label{theo:main1}
The determinant of $A_X$ above takes any integer when $X$ runs over all toric manifolds over $\vt(I^n)$ $(n\ge 3)$ and is invariant under isomorphisms of the cohomology rings of the toric manifolds. Moreover, unless the determinant is one, the following three statements are equivalent for toric manifolds $X$ and $X'$ over $\vt(I^n)$:
\begin{enumerate}
\item $X$ and $X'$ are diffeomorphic,
\item $H^*(X;\Z)$ and $H^*(X';\Z)$ are isomorphic as graded rings,
\item $\det A_X=\det A_{X'}$.
\end{enumerate}
In particular, there is only one diffeomorphism class for toric manifolds over $\vt(I^n)$ with determinant $q$ for each $q\not=1$.
\end{theo}

\begin{rema*}
By the theorem above, the cohomological rigidity holds for toric manifolds over $\vt(I^n)$ with determinant $q\not=1$. One can see that the cohomological rigidity holds for toric manifolds {over $\vt(I^n)$} with determinant one if and only if it holds for Bott manifolds {of complex dimension $n$}.
\end{rema*}

The classification of our toric manifolds as varieties is the following, {which follows from Theorem~\ref{theo:classification as varieties}, Propositions~\ref{prop:proj1} and~\ref{prop:proj2}}.

\begin{theo} \label{theo:main2}
Let $\VAR^n(q)$ denote {the set of variety isomorphism classes} of toric manifolds over $\vt(I^n)$ $(n\ge 3)$ with determinant $q$. Then we have the following.
\begin{enumerate}
\item $\VAR^n(q)$ consists of a single element for each $q\not=0,1,2$.
\item $\VAR^n(0)$ is parametrized by sequences $(b_1,\dots,b_n)$ of integers with $\sum_{i=1}^nb_i=1$ up to cyclic {permutation}.
\item If $X=X'$ in $\VAR^n(1)$, then the Bott manifolds {corresponding to $X$ and $X'$} are isomorphic as varieties.
\item $\VAR^n(2)$ is parametrized by sequences of $1$ and $-1$ of length $n$ with an odd number of $1$'s up to cyclic {permutation}.
\end{enumerate}
Moreover, all elements in $\VAR^n(q)$ are projective when $q\not=2$, and one element in $\VAR^n(2)$ is projective while the others in $\VAR^n(2)$ are non-projective.
\end{theo}

\begin{rema*}
Oda's $3$-fold lies in $\VAR^3(2)$. By (4) above, the cardinality of $\VAR^n(2)$ is equal to the number of binary necklaces of length $n$ with an odd number of zeros, which has been studied in combinatorics. See~\cite[A000016]{oeis}. {Indeed,} this number is known as
\[
\frac{1}{2n}\sum_{\substack{d|n\\ d:\text{odd}}}\varphi(d)2^{n/d}
\]
where $\varphi$ denotes Euler's totient function. {It approaches infinity as $n$ approaches infinity.} 
\end{rema*}

Combining the two theorems above, we obtain two interesting corollaries.

(1) {The} toric manifold over $\vt(I^n)$ with determinant $2$ as a smooth manifold admits a projective variety structure {and} a non-projective variety structure.

(2) We say that a toric manifold $X$ is \textit{cohomologically super-rigid} if any toric manifold $Y$ such that $H^*(Y;\Z)\cong H^*(X;\Z)$ as graded rings is isomorphic to $X$ as {a variety}. In general, it is not true that if $X$ is over $P$ and $H^*(Y;\Z)\cong H^*(X;\Z)$ as graded rings, then $Y$ is also over $P$, but this is true when $P$ is $\vt(I^n)$, {see \cite[Section 6]{ch-pa-su10}.\footnote{{In~\cite{ch-pa-su10}, they assumed that the orbit space of $Y$ is a simple polytope, but their proof is applicable when the orbit space of $Y$ is the dual of a simplicial sphere.}}} Therefore, it follows from the two theorems above that a toric manifold over $\vt(I^n)$ with determinant not equal to $0,1,2$ is cohomologically super-rigid.

This paper is organized as follows. We recall some basic facts about toric varieties in Section~\ref{sec:Preliminaries}. In Section~\ref{sec:Classification of fans}, we investigate the complete {nonsingular} fans associated with toric manifolds over $\vt(I^n)$. We divide the family of the fans into four types and classify the fans up to isomorphism, which proves Theorem~\ref{theo:main2} except {for} the projectivity. We study general properties of the cohomology rings of our toric manifolds in Section~\ref{sec:Properties of cohomology rings of the toric manifolds.}, and then study isomorphism classes of the cohomology rings in each type in Sections~\ref{sec:On cohomology rings of the toric manifolds in each type}. These observations {lead to} the invariance of our determinant under cohomology ring isomorphisms mentioned in Theorem~\ref{theo:main1}. In Sections~\ref{sec:proof of diffeo type 0} and~\ref{sec:proof of diffeo type 2}, we discuss smooth classification of our toric manifolds with determinant $0$ and $2$ respectively. We use the quotient construction of toric manifolds in the former case while we use moment-angle manifolds developed in toric topology in the latter case. Section~\ref{sec:Projectivity} is devoted to the determination of (non-)projectivity of our toric manifolds. {In Appendix, we give a proof to a key proposition used in Section~\ref{sec:proof of diffeo type 2} about moment-angle manifolds and finish with some remarks.}

{Throughout this paper, all cohomology groups will be taken with integer coefficients unless otherwise stated.}

\section{Preliminaries} \label{sec:Preliminaries}
In this section, we prepare some notations and {recall some} basic facts about toric geometry. For more details we refer the reader to  \cite{Bu-Pa15}, \cite{CLS2011}, \cite{fult93}, \cite{oda88}.

\medskip

\noindent\textbf{Toric varieties and Fans.}
A \emph{ toric variety} is a normal variety $X$ that contains an algebraic torus $(\C^*)^n$ as a dense open subset, together with an action $(\C^*)^n\times X\to X$ of $(\C^*)^n$ on $X$ that extends the natural action of $(\C^*)^n$ on itself. A {complete nonsingular} toric variety is called a \emph{toric manifold} in this paper.

A fundamental result of toric geometry is that there is a bijection between toric varieties of complex dimension~$n$ and rational fans of real dimension~$n$, {and toric manifolds correspond to complete nonsingular fans. More strongly, the category of toric varieties is equivalent to the category of fans.}

Throughout this paper, toric varieties we consider {are} toric manifolds, so we focus on simplicial fans.

Let $\cK$ be a simplicial complex with $m$ vertices. We identify the vertex set of $\cK$ with the index set $[m]:=\{1,\ldots,m\}$. Consider a map
\[
\cV:[m]\to\Z^n.
\]
{Henceforth we will denote $\cV(i)$ by $\v_i$.}
The pair $\Delta=(\cK,\cV)$ is called a \emph{(simplicial) fan} of dimension $n$ if it satisfies the following:
\begin{enumerate}
\item $\v_i$'s for $i\in I$ are linearly independent (over $\R$) whenever $I\in\cK$
\item $\cone(I)\cap\cone(J)=\cone(I\cap J)$ for $I,J\in\cK$, where $\cone(I)$ is the cone spanned by $\v_i$'s for $i\in I$, that is,
\[
\cone(I)=\left.\left\{\sum_{i\in I}\alpha_i \v_i~~\right|~~\alpha_i\geq0~\text{for all}~i\in I\right\}.
\]
\end{enumerate}
{The simplicial complex $\cK$ is called the underlying simplicial complex of the fan $\Delta$, and we also say that $\Delta$ is over $\cK$.}
{The} fan $\Delta$ is \emph{complete} if $\bigcup_{I\in\cK}\cone(I)=\R^n$ and \emph{nonsingular}
if the $\v_i$'s for $i\in I$ form a part of {a} basis of $\Z^n$ whenever $I\in\cK$. {Two fans} $\Delta$ and $\Delta'$ are \emph{isomorphic} if there is an isomorphism $\kappa:\cK\to\cK'$ {and  $R\in \GL_{n}(\Z)$} such that $\v_{\kappa(i)}'=R \v_{i}$ for every $i\in[m]$.
{When the fan $\Delta$ is complete and nonsingular,} we denote by $X(\Delta)$ the toric manifold associated with $\Delta$. {We say that {$X(\Delta)$ is over $P$ when the orbit space of $X(\Delta)$ by the compact subtorus of $(\C^*)^n$ is homeomorphic to $P$ as a manifold with corners}.

\medskip
\noindent\textbf{Quotient construction of toric manifolds.}
{Suppose that the fan} $\Delta=(\cK,\cV)$ {is} complete {and} nonsingular. We set
\begin{align*}
Z:=\bigcup_{J\not\in\cK}\{(z_{1},\ldots,z_{m})\in\C^{m}\mid z_{j}=0\text{ for all } j\in J\}, \text{ and }U(\cK):=\C^{m}\setminus Z.
\end{align*}
{The natural action of $(\C^{\ast})^{m}$ on $\C^{m}$ leaves the subset $U(\cK)$ of $\C^{m}$ invariant.
For $\v=[v_{1},\ldots,v_{n}]^T\in\Z^{n}$, we define $\lambda_{\v}\colon \C^*\to (\C^*)^n$ by
\[
\lambda_{\v}(h):=(h^{v_1},\dots, h^{v_n})
\]
and $\lambda_{\cV}:(\C^{\ast})^{m}\to(\C^{\ast})^{n}$ by
\begin{align*}
\lambda_{\cV}(h_{1},\ldots,h_{m}):={\lambda_{\v_1}(h_1)\cdots \lambda_{\v_m}(h_m)}
\end{align*}
where $\v_{i}=\cV(i)\in \Z^n$. The nonsingularity of the fan $\Delta$ implies that $\lambda_{\cV}$ is surjective, so $(\C^{\ast})^{m}/\ker\lambda_{\cV}$ can be identified with $(\C^{\ast})^{n}$ via $\lambda_{\cV}$. Then
\begin{align*}
X(\Delta):=U(\cK)/\ker\lambda_{\cV}
\end{align*}
has an induced effective action of $(\C^{\ast})^m/\ker\lambda_{\cV}\cong (\C^*)^n$ having an open dense orbit and finitely many orbits. Indeed, $X(\Delta)$ with this action of $(\C^{\ast})^{n}$ is the toric manifold associated with the fan $\Delta$.}

\medskip
\noindent\textbf{The cohomology ring of a toric manifold $X(\Delta)$.}
{For each $i=1,\dots,m$, the subset $D_i\subset X(\Delta)$ defined by $z_i=0$ is an invariant divisor fixed pointwise by the $\C^*$-subgroup $\lambda_{\v_i}(\C^*)$.
The Poincar\'e dual of $D_i$ gives a cohomology class $\mu_i$ of degree $2$ in {the cohomology} ring $H^{\ast}(X(\Delta))$.}

\begin{theo}[Danilov-Jurkiewicz] \label{theo:Dan-Jur}
The cohomology ring $H^{\ast}(X(\Delta))$ is isomorphic to $\Z[\mu_1,\ldots,\mu_m]/\cI$ as {a} graded ring, where $\cI$ is the ideal generated by the following two types of elements:
\begin{enumerate}
\item $\prod_{i\in I}\mu_i~~~(I\not\in \cK)$, {and}
\item $\sum_{i=1}^m\s<\u,\v_i>\mu_i$ for any $\u\in\Z^n$,
\end{enumerate}
{where $\langle\ ,\ \rangle$ denotes the standard scalar product on $\Z^n$.}
\end{theo}

For a simplicial complex $\cK$ on $[m]$, a subset $I\subset[m]$ is a \emph{minimal non-face} of $\cK$ if $I\not\in\cK$ but all proper subsets $I'$ of $I$ are simplices of $\cK$. It is known that every simplicial complex is uniquely determined by its minimal non-faces. {Relation} (1) in Theorem~\ref{theo:Dan-Jur} {is} determined by the {data} of minimal non-faces of $\cK$.

{In this paper, we focus on toric manifolds over an $n$-cube with one vertex cut. Since they are closely related to toric manifolds over an $n$-cube, we shall briefly review them.}

\medskip
\noindent\textbf{Bott manifolds and their cohomology rings} {(see \cite{gr-ka94}, \cite{ma-pa08} for details).} A toric manifold over an $n$-cube {$I^n$} is known as a Bott manifold {or a Bott tower}. It can be obtained as the total space of an iterated $\CP^{1}$-bundle starting with a point, where each $\CP^{1}$-bundle is {the} projectivization of the Whitney sum of two complex line bundles. Let $\cB_n$ be the simplicial complex on $[2n]$ whose minimal non-faces are $\{i,n+i\}$ for $i=1,\ldots,n$. {Note that $\cB_n$ is isomorphic to the boundary complex of the {simplicial polytope dual to $I^n$}. The fan associated with a Bott manifold is of the form $\Delta=(\cB_n,\cV)$ {such that} the integer square matrix $A$ defined by
\[
(\v_{n+1},\dots,\v_{2n})=-(\v_1,\dots,\v_n)A\qquad (\v_i=\cV(i))
\]
is unipotent lower triangular. Since the fan $\Delta$ is determined by the matrix $A$, we may denote the Bott manifold $X(\Delta)$ by $X(A)$. Then, setting $x_i:=\mu_{n+i}$ for $i=1,\dots,n$ in Theorem~\ref{theo:Dan-Jur}, one can see that
\begin{align}\label{eq: cohomology ring of a Bott mfd}
H^{\ast}(X(A))\cong\Z[x_{1},\ldots,x_{n}]/\Big(x_{i}\big(x_i+\sum_{j=1}^{i-1} a_{ij}x_{j}\big) \mid 1\leq i\leq n\Big)
\end{align}
where $a_{ij}$ denotes the $(i,j)$-entry of $A$.}

{We note that if $\rho$ is an automorphism of $\cB_n$, then $(\cB_n,\cV\circ\rho)$ defines the same fan as $(\cB_n,\cV)$ and the matrix $A$ will change into the matrix $A_\rho$ defined by
\[
(\v_{\rho(n+1)},\dots,\v_{\rho(2n)})=-(\v_{\rho(1)},\dots,\v_{\rho(n)})A_\rho.
\]
The group $\Aut(\cB_n)$ of automorphisms of $\cB_n$ is generated by two types of elements. One is a permutation $\sigma$ on $[n]$. Indeed, since the minimal non-faces of $\cB_n$ are $\{i, n+i\}$ for $i=1,\dots,n$, $\sigma$ induces an automorphism of $\cB_n$ by sending $n+i$ to $n+\sigma(i)$. The other is a permutation on $[2n]$ which sends $i$ to either $i$ or $n+i$ for each $i=1,\dots,n$.}

\section{Classification of fans} \label{sec:Classification of fans}
{Let $\vcd$ be} a simplicial complex on $[2n+1]$ $(n\ge 2)$ whose minimal non-faces are $\{i,n+i\}$, $\{i,2n+1\}$ $(i=1,\ldots,n)$, and $\{n+1,\ldots,2n\}$. Note that $\vcd$ is isomorphic to the boundary complex of the simplicial polytope {dual to $\vt(I^n)$} an $n$-cube with one vertex cut, see Figure~\ref{fig:vc3}. Therefore, {the fan associated with a toric manifold over $\vt(I^n)$ is over $\vcd$. Conversely, it turns out that the toric manifold associated with a fan over $\vcd$ is over $\vt(I^n)$, see Remark~\ref{rema:over vcd} in Appendix.}

\begin{figure}[h]
\begin{subfigure}[c]{0.45\textwidth}
\begin{center}
\begin{tikzpicture}[scale=0.5]
    \filldraw[inner color=yellow, outer color=yellow!50] (0,0)--(3,0)--(4.5,2)--(4.5,5)--(1.5,5)--(0,3)--cycle;
    \draw (0,3)--(2,3)--(3,2)--(3,0);
    \draw (2,3)--(3.5,11/3)--(3,2);
    \draw (3.5,11/3)--(4.5,5);
    \draw[dotted] (1.5,5)--(1.5,2)--(0,0);
    \draw[dotted] (1.5,2)--(4.5,2);
    \draw (2.8,2.8) node{$7$};
    \draw (1.5,1.5) node{$5$};
    \draw (3.75,2.5) node{$4$};
    \draw (2.5,4) node{$6$};
    \draw (2.15,1) node{\textcolor{gray}{$3$}};
    \draw (2,3.5) node{\textcolor{gray}{$2$}};
    \draw (0.75,2.5) node{\textcolor{gray}{$1$}};
\end{tikzpicture}
\end{center}
\end{subfigure}
\begin{subfigure}[c]{0.45\textwidth}
\begin{center}
\begin{tikzpicture}[scale=.5]
    \filldraw[fill=yellow!50] (0,0)--(2.5,-3)--(5,1)--(2.5,3)--cycle;
    \draw (0,0)--(3,-1)--(5,1);
    \draw (2.5,3)--(3,-1)--(2.5,-3);
    \draw (3,-1)--(3.35,1.25)--(2.5,3);
    \draw (3.35,1.25)--(5,1);
    \draw[dotted] (0,0)--(2,2)--(5,1);
    \draw[dotted] (2.5,3)--(2,2)--(2.5,-3);
    \draw (-0.3,0) node{$1$};
    \draw (2.5,-3.3) node{$3$};
    \draw (5.3,1) node{$4$};
    \draw (2.5,3.3) node{$6$};
    \draw (3.1,-1.3) node{$5$};
    \draw (2,2) node{$2$};
    \draw (3.55,0.95) node{$7$};
\end{tikzpicture}
\end{center}
\end{subfigure}
\caption{$\mathrm{vc}(I^3)$ and $\mathcal{C}_3=\partial(\mathrm{vc}(I^3))^\ast$}\label{fig:vc3}
\end{figure}

\begin{rema} \label{rema:n3}
When $n\ge 3$, the vertex of degree $n$ in $\vcd$ is only $\{2n+1\}$, and $\{1,2,\dots,n\}$ is the unique maximal simplex of $\vcd$ which does not {intersect with the link of} the vertex $\{2n+1\}$. Therefore, any automorphism of $\vcd$ fixes $\{2n+1\}$ and preserves $\{1,\dots,n\}$. Since the minimal non-faces of $\vcd$ are $\{i,n+i\}$ for $i=1,\dots,n$, any automorphism of $\vcd$ is induced from a permutation on $[n]$ by sending $n+i$ to $n+\sigma(i)$ when $n\ge 3$.
\end{rema}
Let $(\vcd,\cV)$ be a complete nonsingular fan.  Since $\{1,\dots,n\}$ is a simplex of $\vcd$, $\{\v_1,\dots,\v_n\}$ is a basis of $\Z^n$.
We express the remaining vectors $\v_{n+1},\dots,\v_{2n}$, and $\v_{2n+1}$ as linear combinations of the basis $\{\v_1,\dots,\v_n\}$ as
\begin{equation} \label{eq:classAb}
(\v_{n+1},\dots,\v_{2n})=-(\v_1,\dots,\v_n)A,\text{ {and} } \v_{2n+1}=-(\v_1,\dots,\v_n)\b,
\end{equation}
where $A$ is an integer square matrix of size $n$ and $\b$ is an integer column vector of size $n$.
{The matrix $A$ and the vector $\b$ depend on a permutation $\sigma$ on $[n]$ as follows. Since $\{\v_{\sigma(1)},\dots,\v_{\sigma(n)}\}$ is a basis of $\Z^n$, an integer square matrix $A_\sigma$ and an  integer column vector $\b_\sigma$ are defined by
\begin{equation} \label{eq:classAbsigma}
\begin{split}
(\v_{n+\sigma(1)},\dots,\v_{n+\sigma(n)})&=-(\v_{\sigma(1)},\dots,\v_{\sigma(n)})A_\sigma,\text{ and}\\ \v_{2n+1}&=-(\v_{\sigma(1)},\dots,\v_{\sigma(n)})\b_\sigma,
\end{split}
\end{equation}respectively.
Therefore, if $P_\sigma$ denotes the permutation matrix defined by
\[
(\v_{\sigma(1)},\dots,\v_{\sigma(n)})=(\v_1,\dots,\v_n)P_\sigma,
\]
then
\[
(\v_{n+\sigma(1)},\dots,\v_{n+\sigma(n)})=(\v_{n+1},\dots,\v_{2n})P_\sigma
\]
and plugging these into \eqref{eq:classAbsigma} and comparing the resulting {equation} with \eqref{eq:classAb}, we see that
\begin{equation} \label{eq:Asigma}
A_\sigma=P_{\sigma}^{-1}AP_\sigma,
\end{equation}in other words, the $(i,j)$-entry of $A_\sigma$ is equal to the $(\sigma(i),\sigma(j))$-entry of $A$. Combining~\eqref{eq:Asigma} with Remark~\ref{rema:n3} shows that when $n\ge 3$, the fan $\Delta$ determines the matrix $A$ up to conjugation {by} permutation matrices, in particular, the value of $\det A$ is an invariant of the fan $\Delta$ when $n\ge 3$.
One also notes that
\begin{equation*}
\text{the $i$th entry of $\b_\sigma$ is equal to the $\sigma(i)$th entry of $\b$.}
\end{equation*}
Motivated by this observation, we make the following definition.

\begin{defi} \label{defi:conjugate}
A pair $(A',\b')$ is conjugate to $(A,\b)$ if there is a permutation $\sigma$ on $[n]$ such that $(A',\b')=(A_\sigma,\b_\sigma)$.
\end{defi}

For a subset $I$ of $[n]$, we define
\[
i(I)= \begin{cases} n+i \quad&\text{if $i\in I$},\\
i\quad &\text{if $i\notin I$}.\end{cases}
\]
If $I\not=[n]$, then ${\widetilde{I}:=}\{1(I),\dots,n(I)\}$ is a simplex of $\vcd$, so $\{\v_{1(I)}, \dots,\v_{n(I)}\}$ forms a basis of $\Z^n$. Therefore, the matrix $A_I$ defined by
\[
(\v_{1(I)},\dots,\v_{n(I)})=-(\v_1,\dots,\v_n)A_I
\]
is unimodular when $I$ is a proper subset of $[n]$. If $I{\subsetneq} J\subsetneq[n]$ and $J\backslash I$ consists of only one element, then {$\cone(\widetilde{I})$ and $\cone(\widetilde{J})$} are adjacent, i.e., their intersection {$\cone(\widetilde{I}\cap\widetilde{J})$} is a codimension one face of each. This implies that $\det A_I$ and $\det A_{J}$ have different signs. We note $A_\emptyset=-E_n$ where $E_n$ denotes the identity matrix of size $n$, and if $I$ consists of a single element $\{i\}$, then $A_{\{i\}}$ is $-E_n$ with the {$i$th column} replaced by the {$i$th column} of $A$. Since $\det A_\emptyset=(-1)^n$ and $\det A_{\{i\}}=(-1)^{n-1}a_{ii}$, {where $a_{ii}$ is the $(i,i)$-entry of $A$,} and they have different signs, we have $a_{ii}=1$. In general, $\det A_I$ is the principal minor of $A$ associated with $I$ multiplied by $(-1)^{n-|I|}$. Then an inductive argument on $|I|$ shows that all proper principal minors of $A$ are $1$. This observation reminds us of the following lemma.
}

\begin{lemm}\cite[Lemma 3.3]{ma-pa08}\label{lem: proper principal minor}
{Let $R$ be a commutative ring with unit $1$ and $A$ be a square matrix of size $n$ with entries in $R$.}
Suppose that every proper principal minor of $A$ is equal to $1$.
If $\det A=1$, then A is conjugate by a permutation matrix to a unipotent lower triangular matrix, and otherwise to a matrix of the form
\begin{align}\label{eq: matrix A when det A is not equal to 1}
\begin{bmatrix}
1 & 0 & \dots &0& a_1 \\
a_2 & 1 & \dots &0& 0\\
\vdots& \vdots &\ddots &\vdots & \vdots \\
0 & 0 & \dots & 1 & 0\\
0 & 0 & \dots & a_n & 1
\end{bmatrix}
\end{align}
where $\det A=1+(-1)^{n-1}\prod_{i=1}^na_i$ and all $a_i$'s are nonzero since $\det A\not=1$.
\end{lemm}

{Applying Lemma~\ref{lem: proper principal minor} with $R=\Z$ to our matrix $A$, we may assume that $A$ is of the form in Lemma~\ref{lem: proper principal minor}. {Based on this understanding, we can characterize} the pairs $(A,\b)$ obtained from complete nonsingular fans over $\vcd$. }

\begin{prop}\label{prop: classification of A and b}
{Let $\vcd$ be as above and $(A,\b)$ be a pair which defines a fan over $\vcd$ by \eqref{eq:classAb}. Then fans over $\vcd$ are classified into four types according to the values of $\det A$. In the following, $\a_i$ denotes the $i$th column vector of $A$ for $i=1,\dots,n$.}

\bigskip
\noindent
{\bf Type 0} $(\det A=0)$
\begin{equation*}
\begin{bmatrix}
1 & 0 & \dots &0& -1 \\
-1 & 1 & \dots &0& 0\\
\vdots& \ddots &\ddots &\vdots & \vdots \\
0 & 0 & \dots & 1 & 0\\
0 & 0 & \dots & -1 & 1
\end{bmatrix},\qquad
\b=
\begin{bmatrix}
b_1\\
b_2\\
\vdots\\
\vdots\\
b_n
\end{bmatrix}
~\text{with}~\sum_{i=1}^nb_i=1.
\end{equation*}

\medskip
\noindent
{\bf Type 1} $(\det A=1)$
\begin{equation*}
\begin{bmatrix}
1 & 0 & \dots &0& 0 \\
a_{2,1} & 1 & \dots &0& 0\\
\vdots& \vdots &\ddots &\vdots & \vdots\\
a_{n-1,1} & a_{n-1,2} & \dots & 1 & 0&\\
a_{n,1} & a_{n,2} & \dots & a_{n,n-1} & 1
\end{bmatrix},\qquad\b=\sum_{i=1}^n\a_i.
\end{equation*}

\medskip
\noindent
{\bf Type 2} $(\det A=2)$
\begin{equation*}
\begin{bmatrix}
1 & 0 & \dots &0& a_1 \\
a_2 & 1 & \dots &0& 0\\
\vdots& \vdots &\ddots &\vdots & \vdots \\
0 & 0 & \dots & 1 & 0\\
0 & 0 & \dots & a_n & 1
\end{bmatrix},\qquad\b=\frac{1}{2}\sum_{i=1}^n\a_i,
\end{equation*}
where $a_i=\pm 1$ {for $i=1,\dots,n$} and the number of $a_i$'s equal to $1$ is odd.

\medskip
\noindent
{\bf Type 3} $(\det A\not=0,1,2)$
\begin{equation*}
\begin{bmatrix}
1 & 0 & \dots &0& a \\
-1 & 1 & \dots &0& 0\\
\vdots& \vdots &\ddots &\vdots & \vdots \\
0 & 0 & \dots & 1 & 0\\
0 & 0 & \dots & -1 & 1
\end{bmatrix}\qquad a\not=0,\pm 1,
\qquad
\b=\frac{1}{\det A}\sum_{i=1}^n\a_i
=
\begin{bmatrix}
1\\
0\\
\vdots\\
\vdots\\
0
\end{bmatrix}.
\end{equation*}
\end{prop}

\begin{rema*}
A fan of Type 1 is a blow-up of a fan of a Bott manifold explained in the last section.
The matrix $A$ with {$a=-1,\,0,\text{ and }1$} in Type 3 is of Type 0, 1, and 2, respectively.
\end{rema*}

\begin{proof}[Proof of Proposition~\ref{prop: classification of A and b}]
Since each $n$-dimensional cone containing $\b$ is nonsingular, the determinant of the matrix obtained from $A$ by replacing $\a_{i}$ with $\b$ is $1$ {up to sign  but we {can} see that it is actually $1$ by taking orientations into account}, that is,
\begin{align}\label{eq:det a1...b...an is equal to 1}
\det[\a_1,\ldots,\a_{i-1},\b,\a_{i+1},\ldots,\a_n]=1,\qquad\text{for}~i=1,2,\ldots,n.
\end{align}
Note that the equality above is equivalent to $\sum_{k=1}^{n}b_{k}A_{ki}=1$, where $A_{ki}$ is the $(k,i)$ minor of $A$ {multiplied by $(-1)^{k+i}$}. Hence,
{if we denote by $\tilde A$ the square matrix whose $(i,k)$-entry is $A_{ki}$, then}
we get
\begin{align}\label{eq:Ab is equal to 1}
\tilde{A}\b=\1,
\end{align}
where
$\1=[1,\ldots,1]^T$. Note that $A\tilde{A}=(\det A)E_{n}$, where $E_{n}$ is the identity matrix of size $n$.

Suppose that $\det A=0$. Then
{$A$ is of \eqref{eq: matrix A when det A is not equal to 1} and}
$A\tilde{A}$ is the zero matrix of size $n$, so if we denote by $\r_i$ the $i$th row vector of $\tilde A$, then
we have
\begin{align*}
\r_{i}+a_{i}\r_{i-1}=0\quad\text{for }1\leq i\leq n,
\end{align*}
where $\r_{0}=\r_{n}$. Hence by adding $a_{i}$ times the $(i-1)$th row to the $i$th row {in} $[\tilde{A},\1]$ from $i=n$ to $i=2$, we get
\begin{align*}
\begin{bmatrix}
1&\ast&\ldots&\ast&1\\
0&0&\ldots&0&1+a_2\\
\vdots&\vdots&{\ddots}&\vdots&\vdots\\
0&0&\ldots&0&1+a_n
\end{bmatrix},
\end{align*}
where $\ast$'s are some {integers. On the other hand, since $\b$ is a solution of $\tilde{A}\x=\1$, we have $\rank[\tilde{A},\1]=\rank\tilde{A}$. Therefore}
$a_{2}=\cdots=a_{n}=-1$. Since $\det A=1+(-1)^{n-1}\prod_{i=1}^n a_i=0$, we also have $a_1=-1$. Hence, $\det[\a_{1},\ldots,\a_{n-1},\b]=\sum_{i=1}^{n}b_{i}$ and $\sum_{i=1}^{n}b_{i}=1$ by \eqref{eq:det a1...b...an is equal to 1}.

We now turn to the case $\det A\neq0$. Since $A\tilde{A}=(\det A) E_{n}$, we obtain $\b=\frac{1}{\det A}\sum_{i=1}^n\a_i$ from \eqref{eq:Ab is equal to 1} regardless of the types.
{When $\det A=1$, this together with Lemma~\ref{lem: proper principal minor} establishes Type 1 case.}

In Type 2, since $\det A=1+(-1)^{n-1}\prod_{i=1}^na_i=2$ {and $a_i$'s are integers}, we can see that $a_i=\pm 1$ {for $i=1,\dots,n$} and the number of $a_i$'s equal to $1$ is odd.

{In} Type 3, we prove the following by using induction {on} $n$;
\begin{quote}
$(\ast)$~The number of $a_i$'s equal to {$-1$} is $n-1$ if $\det A\neq0,1,2$.
\end{quote}
When $n=2$, {we have}
\[
[\v_1,\v_2,\v_3,\v_5,\v_4]={[-\e_1,-\e_2,\a_1,\b,\a_2]=}
\begin{bmatrix}
-1&0&1&b_1&a_1\\
0&-1&a_2&b_2&1
\end{bmatrix},
\]
{and}
\begin{align*}\label{eq: nonsingularity for n=2}
\det[\a_1,\b]=b_2-a_2b_1=1\text{ and }\det[\b,\a_2]=b_1-a_1b_2=1
\end{align*}
from \eqref{eq:det a1...b...an is equal to 1}. This means that $\v_1,\v_2,\v_3,\v_5,\v_4$ are arranged in counterclockwise, and any consecutive two vectors in them form a basis of {$\Z^2$.}  Therefore there exist integers $c_{1},\ldots,c_{5}$ satisfying the {equations}
\begin{align*}
&\v_1+\v_3=c_2\v_2,~\v_2+\v_5=c_3\v_3,\\
&\v_3+\v_4=c_4\v_5,~\v_5+\v_1=c_5\v_4,~\v_4+\v_2=c_1\v_1,
\end{align*}
{and an elementary computation shows that those integers} are as follows:
\begin{equation} \label{eq:cis}
c_1=-a_1,~c_2=-a_2,~c_3=b_1,~c_4=1-a_1a_2,~c_5=b_2.
\end{equation}
{The vectors} $\v_1,\v_2,\v_3,\v_5, \v_4$ are {primitive,} arranged in counterclockwise, and form a $2$-dimensional complete nonsingular fan, {so} the integers $c_{1},\ldots,c_{5}$ {must} satisfy the {equation}\footnote
{If primitive integer vectors $\w_1,\w_2,\ldots,\w_{d}=\w_{0}$ in $\Z^2$ $(d\ge 3)$ are arranged in counterclockwise and form a $2$-dimensional complete nonsingular fan, then we must have $\w_{i-1}+\w_{i+1}=c_{i}\w_{i}$ {with some integers $c_{i}$ for $1\leq i\leq d$} and $\sum_{i=1}^{d}c_{i}=3d-12$. See~\cite[pages 42--44]{fult93}.}
\begin{align*}
\sum_{i=1}^5 c_i=3\times5-12.
\end{align*}
{This together with \eqref{eq:cis} shows that}
$a_1a_2(a_1+1)(a_2+1)=0$.
Combining this equality with the condition $\det A=1-a_1a_2\neq0,1,2$, we see that exactly one of $a_1$ and $a_2$ is equal to $-1$, and the other is different from $0,\pm1$. Hence, statement~$(\ast)$ holds for $n=2$.

Assume that statement $(\ast)$ holds for $n-1$ for $n\geq3$. Let $\Delta_1$ be the fan obtained {from $\Delta$} by projecting onto $\Z^n/\s<\a_1>$. Then $\Delta_{1}$ is determined by {the} $(n-1)\times (n-1)$ matrix
\begin{equation*}\label{page:projection}
A_1=
\begin{bmatrix}
1&0&\ldots&0&-a_1a_2\\
a_3&1&\ldots&0 &0\\
\vdots&\vdots&\ddots&\vdots&\vdots\\
0&0&\ldots&1&0\\
0&0&\ldots&a_n&1
\end{bmatrix}.
\end{equation*}
Since $\det A_1=\det A\neq0,1,2$, {it follows} from the induction hypothesis {that} the number of $(-1)$'s in $\{-a_1a_2,a_3,\ldots,a_n\}$ is equal to $n-2$. Repeating this procedure for $\a_2,\ldots,\a_n$, one sees that the number of $(-1)$'s in $$\{a_1,\ldots,a_{i-1},-a_ia_{i+1},a_{i+2},\ldots,a_n\}$$ is equal to $n-2$ for $i=1,\ldots,n$, where $a_{n+1}=a_1$. This {implies} statement~$(\ast)$.

{We note that any cyclic permutation on $[n]$ preserves the form of the matrix $A$ in \eqref{eq: matrix A when det A is not equal to 1} and permutes $a_i$'s cyclically, so we may assume that the entry $a$ $(\not=0,\pm 1)$ in our matrix $A$ is placed in the $(1,n)$-entry through a cyclic permutation. This establishes Type 3 case.}
\end{proof}

{Suppose that $n\ge 3$. Then a pair $(A,\b)$ is associated with a complete nonsingular fan over $\vcd$ through \eqref{eq:classAb} but it is defined up to conjugation (see Definition~\ref{defi:conjugate}) as observed before. By Proposition~\ref{prop: classification of A and b}, we may assume that the pair $(A,\b)$ is one of the form in the proposition. On the other hand, it is not difficult to see that the pair $(A,\b)$ in Proposition~\ref{prop: classification of A and b} defines a complete nonsingular fan over $\vcd$ through \eqref{eq:classAb} up to isomorphism. However, it happens that two different pairs $(A,\b)$ and $(A',\b')$ define isomorphic fans over $\vcd$. For instance, as remarked at the end of the proof of Proposition~\ref{prop: classification of A and b}, any cyclic permutation on $[n]$ preserves the form of the matrix $A$ in \eqref{eq: matrix A when det A is not equal to 1} and permutes $a_i$'s cyclically.

The following theorem classifies complete nonsingular fans over $\vcd$ up to isomorphism, in other words, toric manifolds over {$\mathrm{vc}(I^n)$} up to variety isomorphism, in terms of the pairs $(A,\b)$.}

\begin{theo}\label{theo:classification as varieties}
{Suppose $n\ge 3$. Then pairs $(A,\b)$ and $(A',\b')$ in Proposition~\ref{prop: classification of A and b} define isomorphic fans over $\vcd$ if and only if they are of the same type in Proposition~\ref{prop: classification of A and b} and}
\begin{enumerate}
\item there is an integer $k\in[n]$ such that $b'_i=b_{i+k}$ for every $i\in [n]$ in Type 0;
\item {there is a permutation matrix $P$ such that $A'=P^{-1}AP$ in Type 1;}
\item there is an integer $k\in[n]$ such that $a'_i=a_{i+k}$ for every $i\in [n]$ in Type 2; and
\item $(A,\b)=(A',\b')$ in Type 3,
\end{enumerate}
{where the indices in (1) and (3) above are taken modulo $n$.}
\end{theo}

\begin{proof}
As remarked before, $\{2n+1\}$ is the unique vertex of $\vcd$ {of} degree $n$ since $n\ge 3$, and hence $\{1,\dots,n\}$ is the unique maximal simplex of $\vcd$ which does not {intersect with the link of} the vertex $\{2n+1\}$. Therefore, the pair $(A,\b)$ is associated with a fan $\Delta$ over $\vcd$ {up to conjugation by \eqref{eq:classAb}}, and if two fans $\Delta$ and $\Delta'$ over $\vcd$ are isomorphic, then the induced isomorphism on $\vcd$ must preserve the simplices $\{1,\dots,n\}$, $\{n+1,\dots,2n\}$, and the vertex $\{2n+1\}$. This implies that $\Delta$ and $\Delta'$ are isomorphic if and only if the associated pairs $(A,\b)$ and $(A',\b')$ are conjugate.
In particular, the value of $\det A$ is invariant under isomorphisms of fans, which implies the first statement in the theorem.

If two pairs $(A,\b)$ and $(A',\b')$ of Type 0 in Proposition~\ref{prop: classification of A and b} are conjugate, then there is a permutation $\sigma$ on $[n]$ such that $A'=A_{\sigma}$ and $\b'=\b_{\sigma}$ (see Definition~\ref{defi:conjugate}). Here the $(i,j)$-entry of $A_\sigma$ is the $(\sigma(i),\sigma(j))$-entry of $A$. Since both $A'=A_\sigma$ and $A$ are the matrix of Type~0, $\sigma$ must be a cyclic permutation. Therefore (1) for Type~0 follows because the $i$th entry of $\b_\sigma=\b'$ is {the} $\sigma(i)$th entry of $\b$.

In the other types, the vector $\b$ is determined by $A$. Since $A_\sigma$ is conjugate to $A$ by the permutation matrix associated with the permutation $\sigma$, (2) for Type 1 follows. As for Types 2 and 3, the permutation $\sigma$ must be a cyclic permutation similarly to Type 0. This proves (3) for Type 2. As for Type 3, the cyclic permutation $\sigma$ must be the identity, which proves (4) for Type 3.
\end{proof}

\begin{rema*}
By (3) in Theorem~\ref{theo:classification as varieties}, the number of isomorphism classes of fans of dimension $n$ in Type 2 is the number of ordered $n$ sets of $1$ and $-1$ with {an odd number of} $1$'s up to cyclic {permutation}.  The number is known in combinatorics {(as mentioned in Introduction)} as
\[
\frac{1}{2n}\sum_{\substack{d|n\\ d:\text{odd}}}\varphi(d)2^{n/d}
\]
where $\varphi$ denotes Euler's totient function. Since $\varphi(1)=1$, {the number is greater than or equal to} $2^{n-1}/n$, and the equality holds if and only if $n$ is a power of $2$.
\end{rema*}

\section{{General properties of cohomology rings}} \label{sec:Properties of cohomology rings of the toric manifolds.}
As observed in the previous section, a fan $\Delta=(\vcd,\cV)$ is determined by a pair $(A,\b)$, so we shall denote the toric manifold $X(\Delta)$ by $X(A,\b)$. In this section, we study general properties of the cohomology ring $H^*(X(A,\b))$.

{We set $x_i:=\mu_{n+i}$ for $i=1,\dots,n$ and $x:=\mu_{2n+1}$ in Theorem~\ref{theo:Dan-Jur}, i.e., $x_{1},\ldots,x_{n},x$ correspond to $\a_{1},\ldots,\a_{n},\b$ respectively. Then one can see that}
\begin{align*}\label{eq: cohomology ring of X(A,b)}
H^{\ast}(X(A,\b))\cong\Z[x_{1},\ldots,x_{n},x]/\cI(A,\b),
\end{align*}
where $\cI(A,\b)$ is the ideal generated by $ \prod_{i=1}^{n}x_{i}$,
\begin{equation}\label{eq: homo poly related to -i}
x_{i}(\sum_{j=1}^{n}a_{ij}x_{j}+b_{i}x), {\text{ and }} x(\sum_{j=1}^{n}a_{ij}x_{j}+b_{i}x)\quad (i=1,\ldots,n),
\end{equation}
where $a_{ij}$ denotes the $(i,j)$-entry of $A$. Note that since $a_{ii}=1$, the former relation in \eqref{eq: homo poly related to -i} shows that $x_i^2$ can be expressed as a linear combination of $x_ix_j$ $(1\le j\not=i\le n)$ and $xx_i$ {for $i=1,\dots,n$}.
\begin{lemm}\label{lem: properties of cohomology class of deg 2}
The elements $x_1,\ldots,x_n,x$ of $H^2(X(A,\b))$ satisfy the following.
\begin{enumerate}
\item $xx_1=xx_2=\cdots=xx_n$, and
\item $x^2=-(\det A)xx_i $.
\end{enumerate}
\end{lemm}
\begin{proof}
In this proof, the indices of $x_i$'s are taken modulo $n$ as usual. We take four cases according to the types of $A$.

In Type 0 $(\det A=0)$,
\begin{equation*}
a_{ij}=
\begin{cases} 1&\text{if }i=j,\\
-1&\text{if }j\equiv i-1 \pmod{n},\\
0&\text{otherwise},
\end{cases}
\qquad \text{and}\qquad \sum_{i=1}^{n}b_{i}=1.
\end{equation*}
Hence, we have
\begin{align} \label{eq:3-1}
x(x_i-x_{i-1}+b_ix)=0~\text{for}~i=1,\ldots,n.
\end{align}
Summing up the equations in \eqref{eq:3-1} over $i=1,\dots, n$, we obtain $(\sum_{i=1}^{n}b_{i})x^{2}=0$, and hence $x^{2}=0$ since $\sum_{i=1}^nb_i=1$. This proves (2). Plugging $x^{2}=0$ into \eqref{eq:3-1}, we obtain (1).

In Type 1 $(\det A=1)$, the matrix $A$ is a unipotent lower triangular matrix and the vector $\b$ is the sum of {all} the column vectors of $A$, so $b_i=1+\sum_{j=1}^{i-1}a_{ij}$ for $i=1,\ldots,n$. Hence, it follows from \eqref{eq: homo poly related to -i} that we have
\begin{align*}
x
\left(
x_i+\sum_{j=1}^{i-1}a_{ij}x_j+(1+\sum_{j=1}^{i-1}a_{ij})x
\right)=0\quad\text{ for $i=1,\dots,n$}.
\end{align*}
{Rewriting these equations, we have
\begin{equation*} \label{eq:xxi}
(xx_i+x^2)+\sum_{j=1}^{i-1}a_{ij}(xx_j+x^2)=0\quad \text{for $i=1,\dots,n$}.
\end{equation*}
Taking $i=1$ above, we obtain $x^2+xx_1=0$. Then, taking $i=2$ above and using $x^2+xx_1=0$, we obtain $x^2+xx_2=0$. Repeating this argument, we see that $x^2+xx_i=0$ for any $i$}, which proves both (1) and (2) for Type 1.

In Type 2 $(\det A=2)$, we have
\begin{equation} \label{eq:aij}
a_{ij}=
\begin{cases} 1&\text{if }i=j,\\
a_i&\text{if }j\equiv i-1 \pmod{n},\\
0&\text{otherwise},
\end{cases} \qquad\text{and}\qquad b_i=(1+a_i)/2,
\end{equation}
where $a_i=\pm 1$ {for $i=1,\dots,n$}. We set $I=\{i\mid a_i=1\}$. The cardinality $|I|$ of $I$ is odd from Proposition \ref{prop: classification of A and b}. It follows from \eqref{eq: homo poly related to -i} and \eqref{eq:aij} that
\begin{align}
x(x_{i-1}+x_i+x)&=0 \text{ for }i\in I,\text{ and} \label{eq:3-2} \\
x(-x_{i-1}+x_i)&=0 \text{ for }i\not\in I. \label{eq:3-3}
\end{align}

Assume that $I=[n]$. Then \eqref{eq:3-2} holds for any $i$. Subtracting $x(x_{i-1}+x_{i}+x)=0$ from $x(x_{i}+x_{i+1}+x)=0$, we get $xx_{i-1}=xx_{i+1}$ for any $i=1,\dots,n$. This proves~(1) since $n=|I|$ is odd. Then \eqref{eq:3-2} implies $x^2=-2x_ix$, which is (2).

Now assume that $I=\{i_{1},\ldots,i_{m}\}\neq [n]$, {where $i_1<i_2<\dots<i_m$}. We divide the index set $[n]$ into $I_1,\ldots,I_m$ such that each $I_{k}=\{i_{k},i_k+1,\ldots,i_{k+1}-1\}$ satisfies
\[
a_{i_{k}}=1,~a_{i_{k}+1}=\cdots=a_{i_{k+1}-1}=-1 \quad\text{for $k=1,\dots,m$}.
\]
From~\eqref{eq:3-3}, we have $xx_{i-1}=xx_i$ for each $i\in I_{k}\setminus\{i_k\}$. Hence

\begin{equation} \label{eq:xxik}
xx_{i_k}={xx_{i_k+1}=}\cdots=xx_{i_{k+1}-1}\quad\text{for $k=1,\dots,m$}.
\end{equation}
On the other hand, since $i_k,i_{k+1}\in I$, it follows from \eqref{eq:3-2} that
\begin{equation} \label{eq:x2x}
x^2=-xx_{i_k-1}-xx_{i_k},\quad x^2=-xx_{i_{k+1}-1}-xx_{i_{k+1}}.
\end{equation}
{Since $xx_{i_{k-1}}=xx_{i_k-1}$ and $xx_{i_k}=xx_{i_{k+1}-1}$ from \eqref{eq:xxik},} it follows from \eqref{eq:x2x} that
\[
xx_{i_{k-1}}=xx_{i_{k+1}}\quad\text{for $k=1,\ldots,m$}.
\]
Since $m$ is odd, this shows that $xx_{i_k}$ is independent of $k$ and hence $x^2=-2xx_{i_k}$ by {\eqref{eq:xxik}} and \eqref{eq:x2x}. This implies both (1) and (2).

In Type 3 $(\det A\neq0,1,2)$, we have
\begin{equation*}
a_{ij}=
\begin{cases} 1&\text{if }i=j,\\
a&\text{if }(i,j)=(1,n),\\
-1&\text{if }j=i-1,\\
0&\text{otherwise},
\end{cases} \qquad \text{and}\qquad \b=\begin{bmatrix} 1\\
0\\
\vdots\\
0\end{bmatrix}.
\end{equation*}
Hence, it follows from \eqref{eq: homo poly related to -i} that
\[
x(x_1+ax_n+x)=0\quad\text{ and }\quad x(-x_{i-1}+x_i)=0\quad\text{for}~i=2,\ldots,n.
\]
The latter equations above mean (1). Then the former equation above implies that $x^2=-(1+a)xx_i$ for any $i$, proving (2) since $\det A=1+a$.
\end{proof}

{\begin{coro}\label{coro: cohomology class of deg n}
{In} the same situation as above, the following {statements} {hold}:
\begin{enumerate}
\item $x^{n}$ is $(\det A)^{n-1}$ times a generator of $H^{2n}(X(A,\b))$.
\item $x_ix_j$ $(1\leq i<j\leq n)$ and $xx_1$ form an additive basis of $H^4(X(A,\b))$.
\end{enumerate}
\end{coro}}
\begin{proof}
It follows from Lemma~\ref{lem: properties of cohomology class of deg 2} that $x^n=(-\det A)^{n-1}x\prod_{i=1}^{n-1}x_i$. Here, since $\{2n+1,1,2,\dots,n-1\}$ is an $(n-1)$-simplex of $\cK$ and our fan is nonsingular, $x\prod_{i=1}^{n-1}x_i$ is a generator of $H^{2n}(X(A,\b))$. This proves (1).

It follows from \eqref{eq: homo poly related to -i} and Lemma~\ref{lem: properties of cohomology class of deg 2} that $x_ix_j$ $(1\leq i<j\leq n)$ and $xx_1$ generate $H^4(X(A,\b))$. On the other hand, the rank of $H^4(X(A,\b))$ is $\binom{n}{2}+1$. This can be seen by computing the $h$-vector of {the simplicial complex} $\vcd$. Another way to see this fact is that the connected sum of $(\C P^1)^n$ and $\C P^n$ with reversed orientation is a toric manifold over {$\vt(I^n)$} and it has the same Betti numbers as $X(A,\b)$. Since $\binom{n}{2}+1$ is the number of the generators $x_ix_j$ $(1\leq i<j\leq n)$ and $xx_1$, they must be an additive basis of $H^4(X(A,\b))$.
\end{proof}

{For an element $z$ of $H^2(X(A,\b))$, we define}
\[
\Ann(z)=\{w\in H^2(X(A,\b))\mid zw=0 \text{ in }H^\ast(X(A,\b))\}.
\]
Since $\{i,2n+1\}$ for $i=1,\dots,n$ is a non-face of $\vcd$, {we have $xx_i=0$ and hence} $\Ann(cx)$ is of rank $n$ for a nonzero constant $c$. The following lemma shows that the converse is also true.

\begin{lemm}\label{lem: Annihilator}
For $n\geq3$, if $\Ann (z)$ is of rank $n$, then $z$ is a nonzero constant multiple of $x$
\end{lemm}

\begin{proof}
Set $z=\sum_{i=1}^nc_ix_i+cx$. We will show that $c_{i}=0$ for $i=1,\ldots,n$ when $\Ann(z)$ is of rank $n$. If an element $\sum_{i=1}^{n}d_{i}x_{i}+dx$ {belongs to} $\Ann(z)$, then we have
\begin{align}\label{eq: zw}
\left(
\sum_{i=1}^{n}c_{i}x_{i}+cx
\right)
\left(
\sum_{i=1}^{n}d_{i}x_{i}+dx
\right)
=0.
\end{align}
{In the following we will use the fact that $x_ix_j$ $(1\le i<j\le n)$ and $xx_1$ form an additive basis of $H^4(X(A,\b))$, see Corollary~\ref{coro: cohomology class of deg n} (2).}

In Type 1, the matrix $A$ is a unipotent lower triangular matrix and we have the relations
$
x_{i}^{2}=
-\sum_{j=1}^{i-1}a_{ij}x_{j}x_{i}-b_{i}x_{i}x
$ ($1\leq i \leq n$) from~\eqref{eq: homo poly related to -i}. Plugging this into~\eqref{eq: zw}, we can see that
the coefficient of $x_ix_j$ (${i>j}$) is $c_id_j+c_jd_i-a_{ij}c_id_i$ and
the coefficients of $x_{{2}}x_{{1}},x_{{3}}x_{{2}}$ and $x_{{3}}x_{{1}}$ satisfy
\[
\begin{bmatrix}
c_{2}&c_{1}-c_{2}a_{21}&0\\
0&c_{3}&c_{2}-c_3a_{32}\\
c_{3}&0&c_{1}-c_{3}a_{31}
\end{bmatrix}
\begin{bmatrix}
d_{1}\\
d_{2}\\
d_{3}
\end{bmatrix}
=
\begin{bmatrix}
0\\
0\\
0
\end{bmatrix}.
\]
Since $\Ann(z)$ is of rank $n$, the leftmost matrix above is of rank at most one. Hence, we get $c_{1}=c_{2}=c_{3}=0$. The coefficients of $x_{{i}}x_{{1}}$ and $x_{{i}}x_{{2}}$ for $3\leq i\leq n$ satisfy
\[
\begin{bmatrix}
c_{i}&0&c_{1}-c_{i}a_{i1}\\
0&c_{i}&c_{2}-c_{i}a_{i2}\\
\end{bmatrix}
\begin{bmatrix}
d_{1}\\
d_{2}\\
d_{i}
\end{bmatrix}
=
\begin{bmatrix}
0\\
0\\
\end{bmatrix}
\text{ for }3\leq i\leq n.
\]
Therefore $c_{i}=0$ for $3\leq i\leq n$ by the same reason as above, proving the lemma for Type 1.

In the other types, the matrix $A$ is of the form~\eqref{eq: matrix A when det A is not equal to 1}, and we have relations
\[
x_{i}^{2}=-a_{i}x_{i}x_{i-1}-b_{i}x_{i}x\quad\text{ for }1\leq i\leq n,
\] from~\eqref{eq: homo poly related to -i}.
Plugging {these} into~\eqref{eq: zw}, we can see that the coefficient of $x_ix_{i-1}$ is $c_{i-1}d_i-a_ic_id_i+c_id_{i-1}$ {for $1\leq i \leq n$ and hence we get}
\[
\begin{bmatrix}
c_{n}-c_{1}a_{1}&0&\ldots&0&c_{1}\\
c_{2}&c_{1}-c_{2}a_{2}&\ldots&0&0\\
\vdots&\vdots&\ddots&\vdots&\vdots\\
0&0&\ldots&c_{n-2}-c_{n-1}a_{n-1}&0\\
0&0&\ldots&c_{n}&c_{n-1}-c_{n}a_{n}
\end{bmatrix}
\begin{bmatrix}
d_{1}\\
d_{2}\\
\vdots\\
d_{n-1}\\
d_{n}
\end{bmatrix}
=
\begin{bmatrix}
0\\
0\\
\vdots\\
0\\
0\\
\end{bmatrix}.
\]
Since $\Ann(z)$ is of rank $n$, the rank of the leftmost matrix above is at most {one}. Hence, we get $c_{i}=0$ for $1\leq i\leq n$.
\end{proof}

\section{{Isomorphism classes of cohomology rings in each type}}\label{sec:On cohomology rings of the toric manifolds in each type}

In this section we study isomorphism classes of the cohomology rings $H^*(X(A,\b))$ in each type. It turns out that the value of $\det A$ is preserved under isomorphisms of the cohomology rings and there is only one isomorphism class for each value of $\det A$ unless $\det A=1$ (Proposition~\ref{prop:coho_iso_det}).

\subsection{Type 0}
{The cohomology ring of $X(A,\b)$ in Type 0 is}
\begin{align*}
H^{\ast}(X(A,\b))\cong\Z[x_{1},\ldots,x_{n},x]/\cI(A,\b),
\end{align*}
where $\cI(A,\b)$ is the ideal generated by homogeneous polynomials
\begin{enumerate}
\item $\prod_{j=1}^{n}x_{j}$ and
\item $x_{i}(-x_{i-1}+x_{i}+b_{i}x)$ and $x(-x_{i-1}+x_{i}+b_{i}x)$ for $1\leq i\leq n$,
\end{enumerate}
where $x_{0}=x_{n}$.

\begin{lemm}\label{lem: cohomology ring of Type 0}
Cohomology rings of toric manifolds of Type 0 are isomorphic to each other.
\end{lemm}

\begin{proof}
Given $\b=[b_{1},\ldots,b_{n}]^{T}$, let $\ell$ be the integer {satisfying} $\sum_{i=1}^{n-1}(n-i)b_{i}\equiv\ell \pmod{n}$ and $0\leq\ell \leq n-1$. Let $\b'=[b_{1}',\ldots,b_{n}']^T$ such that
\begin{align*}
b_{n-\ell}'=1 \quad\text{ and }\quad b'_{i}=0\ \text{ for } i\neq n-\ell.
\end{align*}
Then $\sum_{i=1}^{n-1}(n-i){b_{i}'}=\ell$.
{Note that $(A,\b')$ is conjugate to $(A,\b'_\sigma)$ for any cyclic permutation $\sigma$ on $[n]$ since $A_\sigma=A$ in Type 0 (see Definition~\ref{defi:conjugate} for the conjugation). Hence, it suffices to show that $H^{\ast}(X(A,\b))$ is isomorphic to $H^{\ast}(X(A,\b'))$.}

Since $\sum_{i=1}^{n-1}(n-i)b_{i}\equiv {\ell \equiv} \sum_{i=1}^{n-1}(n-i)b'_{i}~\pmod{n}$,
there exists an integer $\alpha$ such that
\begin{align*}
\ds\sum_{i=1}^{n-1}(n-i)b_{i}-\sum_{i=1}^{n-1}(n-i)b'_{i}+n\alpha=0.
\end{align*}
Then the integers $c_{i}$'s defined by
\begin{align*}
c_{i}=\sum_{j=1}^{i}(b_{j}-b'_{j})+\alpha\ \text{ for } 1\leq i\leq n-1\quad \text{ and }\quad c_{n}=\alpha
\end{align*}
{satisfy} that
\begin{align}\label{eq:4-1}
\sum_{i=1}^{n}c_{i}=0\quad\text{ and }\quad c_{i}-c_{i-1}=b_{i}-b'_{i}\ \text{ for }1\leq i\leq n,
\end{align}
where $c_{0}=c_{n}$ {(note that $\sum_{i=1}^nb_i=\sum_{i=1}^nb_i'=1$)}.

We now consider {the} {automorphism $\varphi$ of $\Z[x_{1},\ldots,x_{n},x]$ defined by
\[
\varphi(x_{i})= x_{i}+c_{i}x\quad\text{for $1\leq i\leq n$}\quad \text{and}\quad \varphi(x)=x.
\]
We {shall check} that $\varphi$ {induces} an isomorphism from $H^\ast(X(A,\b'))$ to $H^\ast(X(A,\b))$. First we have
\begin{align*}
\varphi(x_{i}(-x_{i-1}+x_{i}+b'_{i}x))
&=(x_{i}+c_{i}x)(-x_{i-1}+x_{i}+(-c_{i-1}+c_{i}+b'_{i})x)\\
&=(x_{i}+c_{i}x)(-x_{i-1}+x_{i}+b_{i}x)\qquad\text{(by \eqref{eq:4-1})}\\
&=0 \qquad \text{in $H^*(X(A,\b))$}.
\end{align*}
Similarly we have
\[
\varphi(x(-x_{i-1}+x_i+b'_ix))= x(-x_{i-1}+x_{i}+b_{i}x)=0\qquad \text{in $H^*(X(A,\b))$}.
\]
Finally we have
\begin{align*}
\ds\varphi(\prod_{i=1}^{n}x_{i})
&=\prod_{i=1}^{n}(x_{i}+c_{i}x)\\
&=\sum_{i=1}^{n}(c_{i}x\prod_{j\neq i}^{n}x_{j})\qquad \text{($\because \prod_{i=1}^nx_i=0$ and $x^2=0$ by Lemma~\ref{lem: properties of cohomology class of deg 2})}\\
&=(\sum_{i=1}^{n}c_{i})xx_{1}^{n-1}\qquad \text{($\because xx_1=xx_i$ for any $i$ by Lemma~\ref{lem: properties of cohomology class of deg 2})}\\
&=0 \qquad \qquad \qquad \quad\text{ ($\because \sum_{i=1}^nc_i=0$ by \eqref{eq:4-1})}
\end{align*} {in $H^*(X(A,\b))$.}
This proves that $\varphi$ induces a graded ring homomorphism $\hat{\varphi}$ from $H^\ast(X(A,\b'))$ to $H^\ast(X(A,\b))$. Similarly, the inverse of $\varphi$ induces a graded ring homomorphism in the opposite direction and gives the inverse of $\hat{\varphi}$, proving the lemma.}
\end{proof}

\subsection{Type 1}

Recall that a toric manifold $X(A,\b)$ of Type 1 is a blow up of a Bott manifold $X(A)$ at a torus fixed point. {We have
\begin{align*}
H^{\ast}(X(A))\cong\Z[x_{1},\ldots,x_{n}]/{\cI}(A)
\end{align*}
where ${\cI}(A)$ is the ideal generated by
\[
x_{i}(x_{i}+\sum_{j=1}^{i-1}a_{ij}x_{j}) \qquad (1\leq i\leq n)
\]
by \eqref{eq: cohomology ring of a Bott mfd}. On the other hand, as observed in Section~\ref{sec:Properties of cohomology rings of the toric manifolds.} we have
\begin{align*}
H^{\ast}(X(A,\b))\cong\Z[x_{1},\ldots,x_{n},x]/\cI(A,\b),
\end{align*}
where $\cI(A,\b)$ is the ideal generated by homogeneous polynomials
\begin{enumerate}
\item $\prod_{i=j}^{n}x_{j}$ and
\item $x_{i}(x_{i}+\sum_{j=1}^{i-1}a_{ij}x_{j}+b_{i}x),~x(x_{i}+\sum_{j=1}^{i-1}a_{ij}x_{j}+b_{i}x)$ for $1\leq i\leq n$,
\end{enumerate}where $b_{i}=1+\sum_{j=1}^{i-1}a_{ij}$.}
\begin{lemm}\label{lemm: cohomology type 1}
{Suppose $n\ge 3$. Then $H^*(X(A,\b))$ and $H^*(X(A',\b'))$ in Type 1 are isomorphic as graded rings if and only if $H^*(X(A))$ and $H^*(X(A'))$ are isomorphic as graded rings.}
\end{lemm}
\begin{proof}
{Since $X(A,\b)$ is a blow up of $X(A)$ at a point, $X(A,\b)$ is diffeomorphic to the connected sum of $X(A)$ and $\C P^n$ with reversed orientation. Therefore, the \lq\lq if\rq\rq part is obvious.}

{Let $\hat{\varphi}\colon H^*(X(A,\b))\to H^*(X(A',\b'))$ be a graded ring isomorphism. Since $\{x_1,\dots,x_n,x\}$ is an additive basis of $H^2(X(A,\b))$, $\hat{\varphi}$ induces an automorphism $\varphi$ of $\Z[x_{1},\ldots,x_{n},x]$ such that $\varphi(\cI(A,\b))=\cI(A',\b')$. By Lemma~\ref{lem: Annihilator}, $\varphi(x)=x$ up to sign, so $\varphi$ induces an automorphism $\bar\varphi$ of $\Z[x_{1},\ldots,x_{n}]$ such that $\bar{\varphi}(\cI(A))=\cI(A')$. This proves the ``only if'' part of the lemma.}
\end{proof}

It is conjectured that if two Bott manifolds $X(A)$ and $X(A')$ have isomorphic cohomology rings, then they are diffeomorphic, and this conjecture is affirmatively solved when
\begin{enumerate}
\item the complex dimension of $X(A)$ is at most four,
or
\item $H^\ast(X(A);{\mathbb{F}})$ is isomorphic to $H^{\ast}((\CP^{1})^{n};{\mathbb{F}})$ for ${\mathbb{F}}=\Q$ or $\Z/2\Z$,
\end{enumerate}
see {\cite{choi15}, \cite{C-M12}, \cite{C-M-M15}}, \cite{ch-ma-su10}. Combining these results with Lemma~\ref{lemm: cohomology type 1}, we get the following.

\begin{coro}
Two toric manifolds $X(A,\b)$ and $X(A',\b')$ in Type~1 are diffeomorphic if their cohomology rings are isomorphic as graded rings and the corresponding Bott manifolds $X(A)$ and $X(A')$ satisfy one of the conditions above.
\end{coro}

\subsection{Type 2}

{The cohomology ring of $X(A,\b)$ in Type 2 is}
\begin{align*}
H^{\ast}(X(A,\b))\cong\Z[x_{1},\ldots,x_{n},x]/\cI(A,\b),
\end{align*}
where $\cI(A,\b)$ is the ideal generated by homogeneous polynomials
\begin{enumerate}
\item $x_{i}(x_{i}+a_{i}x_{i-1}+b_{i}x),~x(x_{i}+a_{i}x_{i-1}+b_{i}x)$ for $1\leq i\leq n$, and
\item $\prod_{i=1}^{n}x_{i}$,
\end{enumerate}
where $b_{i}=(1+a_{i})/2$.

\begin{lemm}\label{lem: cohomology ring of Type 2}
Cohomology rings of toric manifolds of Type 2 are isomorphic to each other.
\end{lemm}

\begin{proof}
The matrix $A$ in Type 2 is determined by {its entries} $\{a_{1},\ldots,a_{n}\}$. Assume that the number of $a_{i}$'s equal to $-1$, denoted by $m_A$, is greater than one. Since $n-m_{A}$ is odd, it suffices to show that $H^{\ast}(X(A,\b))$ and $H^{\ast}(X(A',\b'))$ are isomorphic as graded rings whenever $m_{A}-m_{A'}=2$.

In the following, {the indices of $a_i$'s are taken modulo $n$ as usual}. Since the number of $a_{i}$'s equal to $1$ is odd, there exist {$i<j$} such that
\begin{align*}
a_{i}=-1,\ a_{i+1}=\cdots=a_{j-1}=1,\ a_{j}=-1,\quad\text{and\quad $j-i$ is even}.
\end{align*}
Then we have
\begin{equation*}
x_k+a_kx_{k-1}+b_kx=\begin{cases}
x_k-x_{k-1}&\text{ for }k=i\text{ or }j,\\
x_k+x_{k-1}+x&\text{ for }i<k<j.
\end{cases}
\end{equation*}
We set
\begin{align*}
a_{k}'=
\begin{cases}
1&\text{for }i\leq k\leq j\\
a_{k}&\text{otherwise}.
\end{cases}
\end{align*}
Consider {the} automorphism $\varphi$ of $\Z[x_{1},\ldots,x_{n},x]$ defined by
\begin{equation}\label{eq:4-2}
\varphi(x_{k})=
\begin{cases}
-(x_{k}+x)& \text{for }i\leq k\leq j-1,\\
x_{k}& \text{otherwise,}
\end{cases}\quad \text{ and }\quad
\varphi(x)= x .
\end{equation}
We will show that $\varphi$ induces a graded ring isomorphism from $H^\ast(X(A',\b'))$ to $H^\ast(X(A,\b))$.
One can easily check that
\[
\varphi(x_k+a_k'x_{k-1}+b_k'x)=\begin{cases} -(x_k+a_kx_{k-1}+b_kx)\quad&\text{$(i\le k\le j-1)$},\\
x_k+a_kx_{k-1}+b_kx\quad&\text{otherwise}.\end{cases}
\]
This together with \eqref{eq:4-2} shows that both $\varphi(x_{k}(x_{k}+a_{k}'x_{k-1}+b_{k}'x))$ and $\varphi(x(x_{k}+a_{k}'x_{k-1}+b_{k}'x))$ are in $\cI(A,\b)$ for $k=1,\dots,n$.

Let us check that $\varphi(\prod_{k=1}^{n}x_{k})=\prod_{k=1}^nx_k$ in $H^*(X(A,\b))$. {Since $j-i$ is even, it follows from \eqref{eq:4-2} that}
\begin{equation} \label{eq:prodxk}
\varphi(\prod_{k=1}^{n}x_{k})=\prod_{k=i}^{j-1}(x_{k}+x)\prod_{\text{otherwise}}x_{k}.
\end{equation}
{Recall that $x^2=-2xx_k$ and $xx_1=xx_k$ for any $k$ by Lemma~\ref{lem: properties of cohomology class of deg 2}. Then in $H^*(X(A,\b))$ we have
\begin{equation*}
\begin{split}
&\prod_{k=i}^{j-1}(x_{k}+x)\\
=&\prod_{k=i}^{j-1}x_{k}+\left({j-i\choose 1}+{j-i\choose 2}(-2)+\cdots+{j-i\choose j-i}(-2)^{j-i-1}\right)xx_{1}^{j-i-1}\\
=&\prod_{k=i}^{j-1}x_{k}-\frac{1}{2}\Big((1+(-2))^{j-i}-1\Big) xx_1^{j-i-1}\\
=&\prod_{k=i}^{j-1}x_k \qquad (\because \text{$j-i$ is even}).
\end{split}
\end{equation*}
This together with \eqref{eq:prodxk} shows that $\varphi(\prod_{k=1}^{n}x_{k})=\prod_{k=1}^n x_k$ in $H^*(X(A,\b))$.}

{Thus, $\varphi$ induces a graded ring homomorphism $\hat{\varphi}\colon H^\ast(X(A',\b'))\to H^\ast(X(A,\b))$. Similarly, the inverse of $\varphi$ induces a graded ring homomorphism in the opposite direction and gives the inverse of $\hat{\varphi}$, proving the lemma.}
\end{proof}

\subsection{Type 3}

{The cohomology ring of $X(A,\b)$ in Type 3 is}
\begin{align*}
H^{\ast}(X(A,\b))\cong\Z[x_{1},\ldots,x_{n},x]/\cI(A,\b),
\end{align*}
where $\cI(A,\b)$ is the ideal generated by homogeneous polynomials
\begin{enumerate}
\item $\prod_{j=1}^{n}x_{j}$, $x_{1}(x_{1}+ax_{n}+x)$, $x(x_{1}+ax_{n}+x)$ and
\item {$x_{i}(x_{i-1}-x_{i}),~x(x_{i-1}-x_{i})$ for $2\leq i\leq n$.}
\end{enumerate}

\begin{prop}\label{prop: cohomology ring of Type 3}
{When $n\ge 3$, $H^*(X(A,\b))$ and $H^*(X(A',\b'))$ in Type 3 are isomorphic as graded rings if and only if $(A,\b)=(A',\b')$. In other words, when $n\ge 3$,} two toric manifolds of Type 3 are isomorphic as varieties if and only if their cohomology rings are isomorphic as graded rings.
\end{prop}

\begin{proof}
The matrix $A$ of Type 3 is determined by $a$, the $(1,n)$-entry, in Proposition~\ref{prop: classification of A and b}. {Let $A$ and $A'$ be the matrices of Type 3 determined by $a$ and $a'$, respectively.}

Suppose that {$a\not=a'$ and there is a graded ring isomorphism $\varphi:H^{\ast}(X(A,\b))\to H^{\ast}(X(A',\b'))$. We shall deduce a contradiction.} Since $\varphi$ is an isomorphism, it follows from Corollary~\ref{coro: cohomology class of deg n} that
$$
|\det A|=|1+a|=|1+a'|=|\det A'|.
$$
Hence, we have $a+a'=-2$. {We express $\varphi(x_1)$ and $\varphi(x_n)$ as
\begin{align*}
\varphi(x_{1})\equiv\sum_{i=1}^{n}r_{i}x_{i}\pmod{x}\qquad\text{ and }\qquad\varphi(x_{n})\equiv\sum_{i=1}^{n}q_{i}x_{i}\pmod{x}
\end{align*}
where $q_i$ and $r_i$ are integers satisfying}
\begin{align}\label{eq: gcd condition}
\gcd(q_{1},\ldots,q_{n})=\gcd(r_{1},\ldots,r_{n})=1.
\end{align}
Since $\varphi(x)=\pm x$ by Lemma~\ref{lem: Annihilator}, we have
\begin{align} \label{eq:0varphi}
{0=}\varphi\big(x_{1}(x_{1}+ax_{n}+x)\big)\equiv\Big(\sum_{i=1}^{n}r_{i}x_{i}\Big)\Big(\sum_{j=1}^{n}(aq_{j}+r_{j}){x_{j}}\Big)\pmod{x}.
\end{align}
{Here
\[
x_1^2\equiv -a'x_nx_1\pmod{x},\quad x_i^2\equiv x_{i}x_{i-1}\pmod{x} \quad(2\le i\le n)
\]
in $H^*(X(A',\b'))$. Plugging these into \eqref{eq:0varphi} and looking at the coefficients of $x_ix_j$ $(n\ge i>j\ge 1)$}, we see that
\begin{equation*}\label{eq:type3-coefficients relations}
\begin{array}{ll}
r_{n}(aq_{1}+r_{1})+ r_{1}(aq_{n}+r_{n})-r_{1}(aq_{1}+r_{1})a'=0,&\\
r_{i}(aq_{i-1}+r_{i-1})+r_{i-1}(aq_{i}+r_{i})+r_{i}(aq_{i}+r_{i})=0&(2\le i\le n),\\
r_{i}(aq_{j}+{r_{j}})+r_{j}(aq_{i}+r_{i})=0& {((i,j)\not=(n,1),\ i-j\ge 2).}
\end{array}
\end{equation*}
Let $p$ be a divisor of the integer $a$. Since $a+a'=-2$, the {equations} above reduce to
\begin{align}
2r_{1}(r_{n}+r_{1})\equiv0\pmod{p},&\label{eq:4-9} \\
r_{i}(2r_{i-1}+r_{i})\equiv 0\pmod{p}& \quad (2\leq i\leq n), \label{eq:4-8}\\
2r_{i}r_{j}\equiv0 \pmod{p}&\quad {((i,j)\not=(n,1),\ i-j\ge 2).}\label{eq:4-10}
\end{align}

{Suppose that $a$ is even. Then $r_{i}$ is even for $2\leq i\leq n$ by \eqref{eq:4-8} and hence $r_1$ is odd by \eqref{eq: gcd condition}. Therefore $2r_1(r_n+r_1)\not\equiv 0\pmod{4}$. This together with \eqref{eq:4-9} shows that $a$ is not divisible by four and hence $a\equiv 2\pmod{4}$. Since $a+a'=-2$, $a'$ is also even and the same argument as for $a$ shows that $a'\equiv 2 \pmod{4}$. However, these contradict the {assumption} $a+a'=-2$.}

{Therefore, both $a$ and $a'$ must be odd. Suppose $|a|\ge 3$ and let $p$ be an odd prime integer which divides $a$}. Then there exists $r_i$ such that $r_{i}\not\equiv 0\pmod{p}$ from~\eqref{eq: gcd condition} and then $r_{i-1}\not\equiv0\pmod{p}$ from \eqref{eq:4-9} and \eqref{eq:4-8}, where $r_0=r_n$. Therefore  $r_{i}\not\equiv0\pmod{p}$ for any $i$. {However, this contradicts \eqref{eq:4-10} when $n\ge 4$. When $n=3$,  {\eqref{eq:4-9} and \eqref{eq:4-8} have} a nontrivial {common} solution only when $p=5$. Since $a$ is odd, it follows that $a=\pm5^{u}$ for some $u\ge 1$. Therefore $|a'|\ge 3$ since $a+a'=-2$. Then the same argument as for $a$ shows that $a'=\pm5^{v}$ for some $v\geq1$. However, these contradict the {assumption} $a+a'=-2$. Thus, $|a|=|a'|=1$. However, this again contradicts $a+a'=-2$ because $a\not=a'$.}

This completes the proof of the proposition.
\end{proof}

{\subsection{Isomorphism classes and determinant}
We have studied isomorphism classes of cohomology rings in each type and end up with the following proposition.
\begin{prop} \label{prop:coho_iso_det}
If $H^*(X(A,\b))$ and $H^*(X(A',\b'))$ are isomorphic as graded rings, then $\det A=\det A'$, and the converse holds unless the value of the determinant is one.
\end{prop}}
\begin{proof}
{We note that the proof of Proposition~\ref{prop: cohomology ring of Type 3} holds for all values of $a$, i.e. even for $a=0, \pm 1$, and the matrix $A$ in Type~3 is of Type 0, 1 and~2 when $a=-1, 0$ and~$1$. We also know that there is only one isomorphism class in the cohomology rings in Type~0 and~2 by Lemmas~\ref{lem: cohomology ring of Type 0} and~\ref{lem: cohomology ring of Type 2}, and if $H^*(X(A,\b))$ and $H^*(X(A',\b'))$ are isomorphic as graded rings, then $|\det A|=|\det A'|$ by Corollary~\ref{coro: cohomology class of deg n}.
Therefore, it suffices to prove that $H^*(X(A,\b))$ for $A$ with $a=-2$ in Type~3 is not isomorphic to any cohomology ring in Type 1.}

{
If $H^*(X(A,\b))$ above is isomorphic to a cohomology ring $H^*(X(A',\b'))$ in Type~1, then the quotient ring $H^*(X(A,\b))/(x)$ is isomorphic to $H^*(X(A',\b'))/(x)$ because any isomorphism sends $x$ to $x$ up to sign by Lemma~\ref{lem: Annihilator}. The element $x_1$ in $H^*(X(A',\b'))/(x)$ is nonzero but its square vanishes. Therefore, it suffices to show that if $\alpha$ is a nonzero degree two element in $H^*(X(A,\b))/(x)$, then $\alpha^2$ does not vanish in $H^*(X(A,\b))/(x)$. Write $\alpha=\sum_{i=1}^n c_ix_i$ with integers $c_i$. Since $x_1^2=2x_nx_1$ and $x_i^2=x_ix_{i-1}$ for $2\le i\le n$ in $H^*(X(A,\b))/(x)$, we have
\[
\begin{split}
(\sum_{i=1}^nc_ix_i)^2&=\sum_{i=1}^nc_i^2x_i^2+2\sum_{1\le j<i\le n}c_ic_jx_ix_j\\
&=2(c_1^2+c_nc_1)x_nx_1+\sum_{i=2}^n(c_i^2+2c_ic_{i-1})x_ix_{i-1}+2\sum_{\substack{(i,j)\not=(n,1)\\ i-j\ge 2}}c_ic_jx_ix_j.
\end{split}
\]
Therefore, if $(\sum_{i=1}^nc_ix_i)^2=0$ in $H^*(X(A,\b))/(x)$, then we have
\[
c_1(c_1+c_n)=0,\quad c_i(c_i+2c_{i-1})=0 \ (2\le i\le n),\quad c_ic_j=0 \ ((i,j)\not=(n,1),\ i-j\ge 2)
\]
since $x_ix_j$ $(1\le j<i\le n)$ are an additive basis of $H^*(X(A,\b))/(x)$. An elementary check shows that the {equations} above have only {the} trivial solution, proving the proposition.}
\end{proof}

\section{Smooth classification in Type 0} \label{sec:proof of diffeo type 0}

We have seen that the cohomology rings of the toric manifolds of Type 0 are isomorphic to each other (Lemma~\ref{lem: cohomology ring of Type 0}). The purpose of this section is to prove the following proposition using the quotient construction of toric manifolds.
\begin{prop} \label{prop:diffeomorphism class of Type 0}
All the toric manifolds of Type 0 are diffeomorphic to each other.
\end{prop}

{Remember} that the matrix $[A\mid\b]$ in Type 0 is of the form
\begin{equation*}
\begin{bmatrix}
1 & 0 & \dots &0& -1 &b_{1}\\
-1 & 1 & \dots &0& 0&b_{2}\\
\vdots& \ddots &\ddots &\vdots & \vdots&\vdots \\
0 & 0 & \dots & 1 & 0& {b_{n-1}}\\
0 & 0 & \dots & -1 & 1&b_{n}
\end{bmatrix}
\text{ with }\sum_{i=1}^nb_i=1.
\end{equation*}
We add the first row to the second row {in} $[-E_{n}\mid A\mid\b]$, and then add the second row to the third row {in} the new matrix. We repeat the process {and end up} with the matrix
\begin{equation} \label{eq:matrix ci}
\begin{bmatrix}
-1 & & & &1 & & & &-1 & c_1 \cr
-1 & -1 & & & &1 & & &-1 & c_2 \cr
\vdots& &\ddots & & & &\ddots & &\vdots & \vdots \\
-1 & -1 & \dots &-1 \qquad & & & & 1 &-1 & c_{n-1} \\
-1 & -1 & \dots &-1 \ -1& & & & & & 1
\end{bmatrix},\quad c_k=\sum_{i=1}^kb_i.
\end{equation}
Since the fan of Type {0} is isomorphic to the fan determined by the matrix in~\eqref{eq:matrix ci}, we {henceforth} consider the toric manifold associated with the matrix in~\eqref{eq:matrix ci}, denoted by $X_{\bf{c}}$, where $\c=[c_{1},\ldots,c_{n-1}]^{T}$.
Note that
$$
\sum_{k=1}^{n-1}c_{k}=\sum_{i=1}^{n-1}(n-i)b_{i}.
$$
On the other hand, for a cyclic permutation $\sigma(i)=i+\ell$ on $[n]$, {we have
\begin{equation*}
\begin{split}
&\sum_{i=1}^{n-1}(n-i)b_{\sigma(i)}-\sum_{i=1}^{n-1}(n-i)b_i\equiv -\sum_{i=1}^nib_{\sigma(i)}-\Big(-\sum_{i=1}^nib_i\Big)\\
\equiv& -\sum_{i=1}^nib_{i+\ell}+\sum_{i=1}^nib_i\equiv -\sum_{j=1}^n(j-\ell)b_j+\sum_{i=1}^nib_i\\
\equiv&\ \ell\sum_{j=1}^nb_j\equiv \ell\pmod{n}.
\end{split}
\end{equation*}}
Hence, we may restrict our attention to the toric manifolds determined by {a} vector~$\c$ satisfying $\sum_{k=1}^{n-1}c_k\equiv 0\pmod{n}$.

Let us construct the toric manifold $X_\c$ associated with $\c$ by using the quotient construction {explained in Section~\ref{sec:Preliminaries}}. {Remember} that $\vcd$ is the simplicial complex whose minimal non-faces are $\{i,n + i\}, \{i,2n + 1\}$ for $1\leq i\leq n$ and $\{n + 1,...,2n\}$. Letting $(z_{1},\ldots,z_{n},w_{1},\ldots,w_{n},w)$ be the coordinates of $\C^{2n+1}$, we {define}
\begin{align*}\label{eq: complements of coordinate arrangements in Cn}
Z=\bigcup_{i=1}^{n}\{z_{i}=w_{i}=0\}\cup\{w_{1}=\cdots=w_{n}=0\}\cup\bigcup_{i=1}^{n}\{z_{i}=w=0\}.
\end{align*}
Let $\lambda_{\c}:(\C^{\ast})^{2n+1}\to(\C^{\ast})^{n}$ be the homomorphism determined by~\eqref{eq:matrix ci}, that is,
\[
\begin{split}
&\lambda_\c(g_1,\dots,g_n,h_1,\dots,h_{n}, h)\\
=&(g_1^{-1}h_1h_n^{-1}h^{c_1},(g_1g_2)^{-1}h_2h_n^{-1}h^{c_2},\dots,(g_1\cdots g_{n-1})^{-1}h_{n-1}h_n^{-1}h^{c_{n-1}}, (g_1\dots g_n)^{-1}h).
\end{split}
\]
Then the kernel of $\lambda_\c$ is given by
\[
\{(g_1,\dots,g_n,g_1h_nh^{-c_1},g_1g_2h_nh^{-c_2},\dots,(g_1\cdots g_{n-1})h_nh^{-c_{n-1}},h_n,h)\mid h=g_1\cdots g_n\}.
\]
Hence, we get
$$
X_{\c}=(\C^{2n+1}\setminus Z)/\ker\lambda_{\bf c}.
$$

We set
\begin{align*} \label{eq: subset of Xc}
X_{\c}^{-}=X_{\c}\cap\{w\neq0\},\quad X_{\c}^{+}=X_{\c}\cap\bigcap_{i=1}^{n}\{z_{i}\neq0\}.
\end{align*}
We have $X_{\c}=X_{\c}^{-}\cup X_{\c}^{+}$.

\begin{proof}[Proof of Proposition~\ref{prop:diffeomorphism class of Type 0}]
It is sufficient to prove that
\begin{center}
$X_{\c}$ is diffeomorphic to $X_{\0}$, where $\0:=[0,\ldots,0]^{T}\in\Z^{n-1}$,
\end{center}
since it means that $X(A,\b)$ is diffeomorphic to $X(A,\b')$, where $\b=[b_{1},\ldots,b_{n}]^{T}$ with $\sum_{i=1}^{n}b_{i}=1$ and $\b'=[0,\ldots,0,1]^{T}$.

We consider {the} diffeomorphism $\varphi_{\c}:X_{\c}^{-}\to X_{\0}^{-}$ given by
\begin{equation*} \label{eq: diffeomorphism phi_b}
\varphi_\c([z_1,\dots,z_n,w_1,\dots,w_n,w]_\c)=[z_1,\dots,z_n,{w}^{c_1}w_1,\dots,{w}^{c_{n-1}}w_{n-1},w_n,w]_{\0}
\end{equation*}
where $[~]_{\c}$ and $[~]_{\0}$ are points in $X_{\c}$ and $X_{\0}$, respectively, and try to extend it to a diffeomorphism from $X_{\c}$ to $X_{\0}$. We note {that} $X_{\c}^{+}$ is diffeomorphic to $\CP^{n-1}\times \C$ via the map
\begin{align*}
&\psi_{\c}([z_{1},\ldots,z_{n},w_{1},\ldots,w_{n},w]_{\c})\\
&=([z_{1}^{-1}(z_{1}\cdots z_{n})^{c_{1}}w_{1},\ldots,(z_{1}\cdots z_{n-1})^{-1}(z_{1}\cdots z_{n})^{c_{n-1}}w_{n-1},w_{n}],(z_{1}\cdots z_{n})^{-1}w).
\end{align*}
Similarly, ${X_{\0}^+}$ is diffeomorphic to $\CP^{n-1}\times \C$ via the map
\begin{align*}
&\psi_{\0}([z_{1},\ldots,z_{n},w_{1},\ldots,w_{n},w]_{\0})\\
&\qquad=([z_{1}^{-1}w_{1},\ldots,(z_{1}\cdots z_{n-1})^{-1}w_{n-1},w_{n}],(z_{1}\cdots z_{n})^{-1}w).
\end{align*}
Therefore, $\psi_{\0}\circ\varphi_{\c}\circ\psi_{\c}^{-1}$ is a self-diffeomorphism of $\CP^{n-1}\times\C^{\ast}$ given by
\begin{align*}
&\psi_{\0}\circ\varphi_{\c}\circ\psi_{\c}^{-1}([w_{1},\ldots,w_{n}],w)\\
=&\psi_{\0}\circ\varphi_{\c}([1,\ldots,1,w_{1},\ldots,w_{n},w]_{\c})\\
=&\psi_{\0}([{1,\dots,1,w^{c_1}w_1,\dots,w^{c_{n-1}}w_{n-1},w_n,w}]_{\0})\\
=&([w^{c_{1}}w_{1},\ldots,w^{c_{n-1}}w_{n-1},w_{n}],w),
\end{align*}
that is,
\begin{align}\label{eq:self-diffeo of Pn times C}
([w_{1},\ldots,w_{n}],w)\longmapsto([w^{c_{1}}w_{1},\ldots,w^{c_{n-1}}w_{n-1},w_{n}],w).
\end{align}
{If this self-diffeomorphism of $\CP^{n-1}\times\C^*$ extends to a self-diffeomorphism of $\CP^{n-1}\times\C$, then we are done. But this is impossible. To avoid this difficulty, we restrict the map \eqref{eq:self-diffeo of Pn times C} to $\CP^{n-1}\times (\C\backslash \Int D^2)$, where $D^2$ is the unit disk of $\C$, and find its extension to a self-diffeomorphism of $\CP^{n-1}\times\C$.}

We regard the map \eqref{eq:self-diffeo of Pn times C} as the homomorphism
\begin{align*}
\rho:\C^{\ast}\to \PU(n),\quad w\mapsto\diag(w^{c_{1}},\ldots,w^{c_{n-1}},1)
\end{align*}
where $\PU(n)$ denotes the quotient of the unitary group $\U(n)$ by its center and $\diag(~)$ means a diagonal matrix. It suffices to show that $\rho$ restricted to $(\C\setminus\Int D^{2})$, denoted {by} $\bar{\rho}$, extends to a continuous map from $\C$ to $\PU(n)$.\footnote{Let $M$ and $N$ be smooth manifolds {and}  $f:M\to N$ be a continuous map. If $f$ is smooth on a closed subset $A$ in $M$, then the {map $f$ restricted to $A$} extends to a smooth map. We refer the reader to \cite[Theorem 2.23]{mukh15} for the proof. We apply this theorem for $M=\C$, $N=\PU(n)$ and $A=\C\setminus\Int D^{2}$.} For every integer~$m$,
$$
\diag(w^{c_{1}},\ldots,w^{c_{n-1}},1)=\diag(w^{c_{1}+m},\ldots,w^{c_{n-1}+m},w^{m})\qquad \text{in $\PU(n)$}.
$$
{Since $\sum_{k=1}^{n-1}c_{k}\equiv 0\pmod{n}$, this implies that the homomorphism $\rho$ above factors through} the special unitary group $\SU(n)$. Since $\SU(n)$ is simply connected, the restriction of $\bar{\rho}$ to the boundary ${S^{1}}$ of $D^2$ is null homotopic. Therefore, $\bar{\rho}$ extends to a map $\C\to\PU(n)$ continuously.
\end{proof}

\section{Smooth classification in Type 2}\label{sec:proof of diffeo type 2}

We have seen that the cohomology rings of the toric manifolds of Type 2 are isomorphic to each other (Lemma~\ref{lem: cohomology ring of Type 2}). The purpose of this section is to prove the following proposition using moment-angle manifolds associated with an $n$-cube with one vertex cut, which is a deformation retract of $\C^{2n+1}\backslash Z$ used in the previous section.

\begin{prop} \label{prop: diffeomorphism class of Type 2}
All the toric manifolds of Type 2 are diffeomorphic to each other.
\end{prop}

First, let us review the definition of the moment-angle manifold $\ZP$ obtained from a simple polytope $P$.
Consider an $n$-dimensional simple polytope
\[
P=\{\x=(x_{1},\ldots,x_{n})\in\R^{n}\mid\s<\n_{i},\x>+\gamma_{i}\geq 0\quad\text{for }i=1,\ldots,m\},
\]
where $\n_i\in \R^n$, $\gamma_i\in \R$ and $\langle\ ,\ \rangle$ denotes the standard scalar product on $\R^n$.
{We assume that $m$ agrees with the number of facets of $P$.} Define the map
\[
\iota_{P}:\R^{n}\to\R^{m},
\quad
\iota_{P}(\x)=(\s<\n_{1},\x>+\gamma_{1},\ldots,\s<\n_{m},\x>+\gamma_{m}).
\]
It embeds $P$ into $\R_{\geq 0}^{m}$. {Then the moment-angle manifold $\ZP$ associated with $P$ is defined to be the fiber product of} the commutative diagram
\begin{equation*}
\begin{CD}
\ZP @>>> \C^m\\
@VVV @VV\pi V\\
P@>\iota_P>> \R^m_{\ge 0}
\end{CD}
\end{equation*}
where $\pi(z_{1},\ldots,z_{m})=(|z_{1}|^{2},\ldots,|z_{m}|^{2})$.
{We note that $\ZP$ is invariant under the standard action of $(S^1)^m$ on $\C^m$.}
The moment-angle manifold associated with an $n$-cube is the product of $n$ copies of $3$-spheres.

Let us consider the polytope $P$ presented as follows:
\begin{equation}\label{eq: polytope}
P=\left\{(x_{1},\ldots,x_{n})\in\R^{n}\,\middle|\, 0\leq x_{i}\leq 1,~\sum_{i=1}^{n}x_{i}\leq n-\frac{1}{2}\right\}.
\end{equation}
Then $P$ is an $n$-cube with one vertex cut {$\vt(I^n)$ as a manifold with corners}, so the boundary complex of the {simplicial polytope dual to} $P$ is {isomorphic to} our simplicial complex $\vcd$. One can see that the moment-angle manifold $\ZP$ associated with $P$ can be described as
\begin{align} \label{eq:ZP}
\left\{(z_{1},\ldots,z_{n},w_{1},\ldots,w_{n},w)\in\C^{2n+1}\,\middle|\, |z_{i}|^{2}+|w_{i}|^{2}=1,~\sum_{i=1}^{n}|w_{i}|^{2}=|w|^{2}+\frac{1}{2}\right\}.
\end{align}
As is well-known, $\ZP$ is a deformation retract of $\C^{2n+1}\backslash Z$, where
$$Z=\bigcup_{i=1}^{n}\{z_{i}=w_{i}=0\}\cup\{w_{1}=\cdots=w_{n}=0\}\cup\bigcup_{i=1}^{n}\{z_{i}=w=0\}$$
as before.

Suppose that $(A,\b)$ is of Type 2. We define {a} homomorphism $\lambda_{A}\colon (\C^{\ast})^{2n+1}\to(\C^{\ast})^{n}$ by
\begin{align*}
\lambda_{{A}}(g_{1},\ldots,g_{n},h_{1},\ldots,h_{n},h) =(g_{1}^{-1}h_{1}h_{n}^{a_{1}}h^{b_1},g_{2}^{-1}h_{1}^{a_{2}}h_{2}g^{b_{2}},\ldots,g_{n}^{-1}h_{n-1}^{a_{n}}{h_n}h^{b_{n}}).
\end{align*}
Then
\begin{align} \label{eq: kernel lambda A}
\ker\lambda_{A}=\{h_{1}h_{n}^{a_{1}}h^{b_1},h_{1}^{a_{2}}h_{2}{h}^{b_{2}},\ldots,h_{n-1}^{a_{n}}{h_n}{h}^{b_{n}},h_{1},\ldots,h_{n},h)\}
\end{align}
and the quotient construction of toric manifolds tells us that
$$X(A,\b)=(\C^{2n+1}\backslash Z)/\ker\lambda_{A}.$$
The following proposition is a real analog of this {quotient} construction. It is well-known for experts, but we shall give a proof in Appendix for the convenience of the reader.

\begin{prop} \label{prop: keyproposition}
The toric manifold $X(A,\b)$ of Type 2 is {$(S^1)^n$-equivariantly} {diffeomorphic} to $\ZP/\ker\lambda_A^S$, where $\lambda_A^S$ is the restriction of $\lambda_A$ to $(S^1)^{2n+1}$.
\end{prop}

We prepare some notations and a lemma we need later. Define the matrix $R$ by
\begin{align}\label{eq: the matrix R}
R=A^{-1}+\frac{1}{2}J,
\end{align}
where $J$ is the $n\times n$ matrix all of whose entries are $1$. One can easily check that every component of $A^{-1}$ is either $1/2$ or $-1/2$, so the matrix $R$ is an integer matrix. Let $\r_{i}$ be the $i$th row of $R$ and $\u_{i}$ be the $i$th row of $A$. We write $\r_{i}=[r_{i1},r_{i2},\ldots,r_{in}]$ and $\u_i=[u_{i1},\ldots,u_{in}]$. Set $\e_{i}'=[0,\ldots,0,1,0,\ldots,0]$, the vector with a $1$ in the $i$th coordinate and $0$'s elsewhere, and $\1:=[1,\ldots,1]$.
\begin{lemm}\label{lem: equalities about r, u}
Under the notations above, we obtain the following:
\begin{enumerate}
\item$\sum_{j=1}^n r_{ij}\u_j=\frac{1}{2}\sum_{j=1}^n\u_j+\e'_{{i}}$;
\item$\sum_{j=1}^n r_{ij}b_j=\frac{1}{2}(1+\sum_{j=1}^nb_j)$; and
\item$\sum_{j=1}^n u_{ij}\r_j-b_i\1=\e'_i$,
\end{enumerate} where $b_i$ is the $i$th entry of $\b$, that is, $b_i=\frac{1}{2}(1+a_i)$ for $1\leq i\leq n$.
\end{lemm}
\begin{proof}
From~\eqref{eq: the matrix R}, we have
\begin{equation}\label{eq:RA}
RA=\frac{1}{2}JA+E_{n}.
\end{equation}
Then we obtain (1) by comparing the rows on both sides of~\eqref{eq:RA}.

Multiplying {both sides of \eqref{eq: the matrix R} by $\b$ from the right}, we have
\begin{align}\label{eq:Rb}
R
\begin{bmatrix}
b_{1}\\
\vdots\\
b_{n}
\end{bmatrix}
=
A^{-1}
\begin{bmatrix}
b_{1}\\
\vdots\\
b_{n}
\end{bmatrix}
+\frac{1}{2}\sum_{j=1}^{n}b_{j}
\begin{bmatrix}
1\\
\vdots\\
1
\end{bmatrix}
\end{align}
Then, {noting $A\1^T=2\b$,} we obtain (2) by comparing the rows on both sides of~\eqref{eq:Rb}.

From~\eqref{eq: the matrix R}, we have
\begin{equation}\label{eq:AR}
AR-\frac{1}{2}AJ=E_{n}.
\end{equation}
Since $\sum_{j=1}^{n}u_{ij}=1+a_{i}=2b_{i}$, we obtain (3) by comparing the rows on both sides of~\eqref{eq:AR}.
\end{proof}

We look at two parts of $\ZP/\ker\lambda_A^S$:
\[
(\ZP/\ker\lambda_{A}^{S})\cap\{|w|<1/\sqrt{2}\}\qquad\text{and}\qquad (\ZP/\ker\lambda_{A}^{S})\cap\{w\neq 0\}.
\]
Note that if $|w|<1/\sqrt{2}$, then $z_{i}\neq0$ for every $i=1,\ldots,n$ by \eqref{eq:ZP}. Write
\begin{align*}
\z^{\c}=\prod_{j=1}^{n}z_{j}^{c_{j}}\text{ and }\h^{c}=\prod_{j=1}^{n}h_{j}^{c_{j}},
\end{align*}
for $\c=[c_{1},\ldots,c_{n}]$.
{It follows from \eqref{eq: kernel lambda A} that}
\begin{align} \label{eq:kerlambdaA}
\ker\lambda_{A}^{S}=\{(\h^{\u_{1}}h^{b_{1}},\ldots,\h^{\u_{n}}h^{b_{n}},h_{1},\ldots,h_{n},h)\in(S^{1})^{2n+1}\mid h_{i},h\in S^{1}\}.
\end{align}
We define
\begin{align*}
\tilde{L}=\left\{(v_{1},\ldots,v_{n},u)\in\C^{{n+1}}~\left|~\sum_{i=1}^{n}|v_{i}|=|u|^{2}+\frac{1}{2},~|u|<\frac{1}{\sqrt{2}}\right\}\right.\text{ and } L=\tilde{L}/S^{1},
\end{align*}
where the action of $S^{1}$ on $\tilde{L}$ is given by
\begin{align} \label{eq: S1-action on tildeL}
g\cdot(v_{1},\ldots,v_{n},u)=(gv_{1},\ldots,gv_{n},g^{-2}u)\text{ for }g\in S^{1}.
\end{align}

\begin{lemm}\label{lem: varphi A induces diffeomorphism}
The map $\varphi_{A}\colon (\C^{\ast})^{n}\times\C^{n+1}\to \tilde{L}$ defined by
\begin{equation}\label{eq: the map induces diffeomorphism}
\varphi_{{A}}(z_{1},\ldots,z_{n},w_{1},\ldots,w_{n},w)
=\Biggl(\biggl(\frac{\z^{\r_{1}}}{|\z^{\r_{1}}|}\biggr)^{-1}w_{1},\ldots,\biggl(\frac{\z^{\r_{n}}}{|\z^{\r_{n}}|}\biggr)^{-1}w_{n},\frac{\z^{\1}}{|\z^{\1}|}w\Biggr)
\end{equation} induces a diffeomorphism $\hat{\varphi}_{A}:(\ZP/\ker\lambda_{A}^{S})\cap\{|w|<1/\sqrt{2}\}\to L$.
\end{lemm}

\begin{proof}
We first show that $\varphi_{A}$ induces a {smooth} map
$$\hat{\varphi}_{A}:(\ZP/\ker\lambda_{A}^{S})\cap\{|w|<1/{\sqrt{2}}\}\to L.$$
{By \eqref{eq:kerlambdaA}} it is sufficient to show that {for each $(h_1,\dots,h_n,h)\in (S^1)^{n+1}$,} there exists $t\in S^1$ such that
\begin{equation} \label{eq: map from ZP/ker to L}
\begin{split}
&\varphi_A(\h^{\u_1}h^{b_1}z_1,\dots,\h^{\u_n}h^{b_n}z_n,h_1w_1,\dots,h_nw_n,hw)\\
&= \left(t\Big(\frac{\z^{\r_1}}{|\z^{\r_1}|}\Big)^{-1}w_1,\dots,t\Big(\frac{\z^{\r_n}}{|\z^{\r_n}|}\Big)^{-1}w_n,t^{-2}\frac{\z^{\mathbf 1}}{|\z^{\mathbf 1}|}w\right).
\end{split}
\end{equation}
Since we have $|h_{i}|=|h|=1$ for every $i=1,\ldots,n$, the $i$th coordinate of the {left hand side of \eqref{eq: map from ZP/ker to L}} is equal to
\begin{align*}
\left(\prod_{j=1}^n (\h^{\u_j}h^{b_j}z_j/|z_j|)^{r_{ij}}\right)^{-1}h_iw_i
&=\left(\h^{\sum_{j=1}^nr_{ij}\u_j} h^{\sum_{j=1}^nr_{ij}b_j}\frac{\z^{\r_i}}{|\z^{\r_i}|}\right)^{-1}h_iw_i\\
&=\left(\h^{\frac{1}{2}\sum_{j=1}^n\u_j}h_ih^{\frac{1}{2}(1+\sum_{j=1}^nb_j)}\frac{\z^{\r_i}}{|\z^{\r_i}|}\right)^{-1}h_iw_i\\
&=\left(\h^{\frac{1}{2}\sum_{j=1}^n\u_j}h^{\frac{1}{2}(1+\sum_{j=1}^nb_j)}\right)^{-1}\left(\frac{\z^{\r_i}}{|\z^{\r_i}|}\right)^{-1}w_i,
\end{align*}
where the second equality follows from (1) and (2) in Lemma~\ref{lem: equalities about r, u}. The {$(n+1)$}th coordinate of the {left hand side of \eqref{eq: map from ZP/ker to L}} is equal to
\begin{align*}
&\left(\prod_{j=1}^n (\h^{\u_j}h^{b_j}z_j/|z_j|)\right)hw=\left(\h^{\sum_{j=1}^n\u_j}h^{1+\sum_{j=1}^nb_j}\right)\frac{\z^{\mathbf 1}}{|\z^{\mathbf 1}|}w.
\end{align*}
Hence,~\eqref{eq: map from ZP/ker to L} holds for $t=\left(\h^{\frac{1}{2}\sum_{j=1}^n\u_j}{h}^{\frac{1}{2}(1+\sum_{j=1}^nb_j)}\right)^{-1}\in S^{1}$.

Let us show the injectivity of $\hat{\varphi}_{A}$. Suppose that
$$
\varphi_{A}(z_{1},\ldots,z_{n},w_{1},\ldots,w_{n},w)=\varphi_{A}(z_{1}',\ldots,z_{n}',w_{1}',\ldots,w_{n}',w')\quad\text{in $L$},
$$
that is, there is an element $t\in S^1$ such that
\begin{align}\label{eq: equality of equivalence class of L}
\begin{split}
&\left(\biggl(\frac{\z'^{\r_{1}}}{|\z'^{\r_{1}}|}\biggr)w_{1}',\ldots,\biggl(\frac{\z'^{\r_{n}}}{|\z'^{\r_{n}}|}\biggr)w_{n}',\biggl(\frac{\z'^{\1}}{|\z'^{\1}|}\biggr)w'\right)\\
=
&\left(t\biggl(\frac{\z^{\r_{1}}}{|\z^{\r_{1}}|}\biggr)w_{1},\ldots,t\biggl(\frac{\z^{\r_{n}}}{|\z^{\r_{n}}|}\biggr)w_{n},t^{-2}\biggl(\frac{\z^{\1}}{|\z^{\1}|}\biggr)w\right).
\end{split}
\end{align}
We set $g_{i}=\Bigl( \frac{z_{i}'}{|z_{i}'|}\Bigr)^{{-1}}\Bigl( \frac{z_{i}}{|z_{i}|}\Bigr)$ {and} $
\g^{\c}=\prod_{j=1}^{n}g_{j}^{c_{j}}$ for $\c=[c_{1},\ldots,c_{n}]$. Then
\begin{align*}
w_{i}'=t\g^{\r_{i}}w_{i}\text{ and }w'=t^{-2}\g^{-\1}w
\end{align*}
from~\eqref{eq: equality of equivalence class of L}. Setting $h_{i}=t\g^{\r_{i}}$ and $h=t^{-2}\g^{-\1}$, we obtain
\begin{align*}
\h^{\u_i}h^{b_i}=\big(\prod_{j=1}^nh_j^{u_{ij}}\big)h^{b_i}=\left(t^{\sum_{j=1}^n{u_{ij}}}\g^{\sum_{j=1}^nu_{ij}\r_j}\right)\left(t^{-2b_i}\g^{-b_i\mathbf 1}\right)=g_i,
\end{align*}
where the last equality {above} follows from $\sum_{j=1}^{n}u_{ij}=1+a_{i}=2b_{i}$ and (3) in Lemma~\ref{lem: equalities about r, u}. Therefore, $(g_{1},\ldots,g_{n},h_{1},\ldots,h_{n},h)\in\ker\lambda_{A}^{S}$, which implies that $\hat\varphi_A$ is injective.

For $(v_{1},\ldots,v_{n},u)\in \tilde{L}$, $(z_{1},\ldots,z_{n},v_{1},\ldots,v_{n},u)$ with $z_{i}=\sqrt{1-|v_{i}|^{2}}$ is an element of $\ZP$ and $\varphi_A(z_{1},\ldots,z_{n},v_{1},\ldots,v_{n},u)=(v_{1},\ldots,v_{n},u)$, which proves that $\hat{\varphi}_{A}$ is surjective.

Thus, $\hat{\varphi}$ is smooth and bijective. The inverse of $\hat{\varphi}$ is induced from {the} map sending $(v_{1},\ldots,v_{n},u)\in \tilde{L}$ to $(z_{1},\ldots,z_{n},v_{1},\ldots,v_{n},u)\in\ZP$ with $z_{i}=\sqrt{1-|v_{i}|^{2}}$, so it is also smooth. Therefore, $\hat{\varphi}$ is a diffeomorphism.
\end{proof}

Suppose that the matrix $A$ has at least {two $(-1)$'s} in $\{a_{1},\ldots,a_{n}\}$. We consider the {indices of $a_i$'s modulo $n$ as usual}. Since the number of $1$'s in $\{a_{1},\ldots,a_{n}\}$ is odd, there {exist indices $i<j$} such that
\begin{align}\label{eq: sequence of ai}
a_{i}=-1,~a_{i+1}=\cdots=a_{j-1}=1,~a_{j}=-1,{\text{ and } \text{ $j-i$ is even}}.
\end{align}
Set
\begin{align}\label{eq: relation of between ak' and ak}
a_{k}'=
\begin{cases}
1&\text{for }i\leq k\leq j\\
a_{k}&\text{otherwise}
\end{cases}
\end{align}
and denote by $A'$ the matrix of Type 2 associated with $\{a_{1}',\ldots,a_{n}'\}$. We will use the prime symbol to represent notations corresponding to $A'$.

\begin{lemm}\label{lem: diffeo ZP cap w nonzeo to ZP cap w' nonzero}
The map $f:\C^{2n}\times \C^\ast\to\C^{2n}\times \C^\ast$ defined by
\begin{equation*}
\begin{split}
&f(z_{1},\ldots,z_{n},w_{1},\ldots,w_{n},{w})\\
&=(z_1,\dots,z_{i-1},\frac{w}{|w|}\bar{z}_i,\left(\frac{w}{|w|}\right)^2\bar{z}_{i+1},\dots,\left(\frac{w}{|w|}\right)^2\bar{z}_{j-1},\frac{w}{|w|}z_j,\\
&\qquad\qquad z_{j+1},\dots,z_n,w_1,\dots,w_{i-1},\bar{w}_i,\dots,\bar{w}_{j-1},w_j,\dots,w_n,w)
\end{split}
\end{equation*} induces a diffeomorphism $\hat{f}:(\ZP/\ker\lambda_{A}^{S})\cap\{w\neq0\}\to(\ZP/\ker\lambda_{A'}^{S})\cap\{w\neq0\}$.
\end{lemm}

\begin{proof}
As is easily checked, $f$ preserves $\ZP$ (see \eqref{eq:ZP}) and is a diffeomorphism, so it suffices to show that $f$ is weakly equivariant with respect to the actions of $\ker\lambda_A^S$ and $\ker\lambda_{A'}^S$.

Remember the description of $\ker\lambda_A^S$ in \eqref{eq:kerlambdaA}. Since the complex conjugate of an  element in $S^{1}$ is equal to its inverse {and} $|h|=1$ for $h\in S^1$, we have
\begin{align}\label{eq: conjugation is equal to inverse}
\begin{split}
&f(\h^{\u_1}h^{b_1}z_1,\dots,\h^{\u_n}h^{b_n}z_n,h_1w_1,\dots,h_nw_n,hw)\\
&=(\h^{\u_1}h^{b_1}z_1,\dots,\h^{\u_{i-1}}h^{b_{i-1}}z_{i-1},\left(\frac{hw}{|w|}\right){\h}^{-\u_{i}}{h}^{-b_{i}}\bar{z}_i, \\
&\qquad \left(\frac{hw}{|w|}\right)^2{\h}^{-\u_{i+1}}{h}^{-b_{i+1}}\bar{z}_{i+1},\dots,\left(\frac{hw}{|w|}\right)^2{\h}^{-\u_{j-1}}{h}^{-b_{j-1}}\bar{z}_{j-1}, \\&\qquad \left(\frac{hw}{|w|}\right)\h^{\u_{j}}h^{b_{j}}z_j,\h^{\u_{j+1}}h^{b_{j+1}}z_{j+1},\dots,\h^{\u_{n}}h^{b_{n}}z_n,\\
&\qquad h_1w_1,\dots,h_{i-1}w_{i-1},{h}_i^{-1}\bar{w}_i,\dots,{h}_{j-1}^{-1}\bar{w}_{j-1},h_jw_j,\dots,h_nw_n,hw).
\end{split}
\end{align}
Here, since $\u_{k}=[0,\ldots,0,a_{k},1,0\ldots,0]$, it follows from \eqref{eq: sequence of ai} that we have
\begin{equation*}
\begin{split}
&{\h}^{-\u_i}=h_{i-1}^{-a_i}h_i^{-1}=h_{i-1}h_i^{-1},\\
&{\h}^{-\u_k}={h}_{k-1}^{-a_k}{h}_k^{-1}=h_{k-1}^{-1}h_k^{-1}\quad(i+1\le k\le j-1),\\
&\h^{\u_j}=h_{j-1}^{a_j}h_j=h_{j-1}^{-1}h_j,
\end{split}
\end{equation*}
and note that $\h^{\u_{k}}$ does not contain $h_{i},h_{i+1},\ldots,h_{j-1}$ unless $i\le k\le j$.
{On the other hand, it follows from the definition} of $a'_k$'s in~\eqref{eq: relation of between ak' and ak} that we have
\begin{align*}
\begin{split}
&\u_k'=\u_k,~b_k'=b_k \quad\text{for }k\not=i,j,\\
&b_k'=b_k=1\quad\text{for }i+1\le k\le j-1,\\
&b_i'=b_j'=1,\quad b_i=b_j=0.
\end{split}
\end{align*}
Hence, the right hand side of~\eqref{eq: conjugation is equal to inverse} is written as
\begin{align} \label{eq:hprime}
\begin{split}
&(\h^{\u_1'}h^{b_1'}z_1,\dots,\h^{\u_{i-1}'}h^{b_{i-1}'}z_{i-1},\left(\frac{w}{|w|}\right)h_{i-1}h_i^{-1}{h}^{b_{i}'}\bar{z}_i, \\
&\qquad \left(\frac{w}{|w|}\right)^2h_{i}^{-1}h_{i+1}^{-1}{h}^{b_{i+1}'}\bar{z}_{i+1},\dots,\left(\frac{w}{|w|}\right)^2h_{j-2}^{-1}h_{j-1}^{-1}{h}^{b_{j-1}'}\bar{z}_{j-1}, \\&\qquad \left(\frac{w}{|w|}\right){h_{j-1}^{-1}}h_jh^{b_{j}'}z_j,\h^{\u_{j+1}'}h^{b_{j+1}'}z_{j+1},\dots,\h^{\u_{n}'}h^{b_{n}'}z_n,\\
&\qquad h_1w_1,\dots,h_{i-1}w_{i-1},{h}_i^{-1}\bar{w}_i,\dots,{h}_{j-1}^{-1}\bar{w}_{j-1},h_jw_j,\dots,h_nw_n,hw).
\end{split}
\end{align}
{We set} $\h'=(h_1,\ldots,h_{i-1},h_i^{-1},\ldots,h_{j-1}^{-1},h_j,\ldots,h_n,h)$. {Then
\[
\begin{split}
&\h^{\u_k'}=(\h')^{\u_k'} \quad \text{(unless $i\le k\le j$)},\\
&h_{i-1}h_i^{-1}=(\h')^{\u_i'},\quad
h_{k}^{-1}h_{k+1}^{-1}=(\h')^{\u_k'} \ \ (i+1\leq k\leq j-1),\quad
h_{j-1}^{-1}h_j=(\h')^{\u_j'}.
\end{split}
\]
This together with \eqref{eq: conjugation is equal to inverse} and \eqref{eq:hprime} shows that $f$ is $\theta$-equivariant with respect to the actions of $\ker\lambda_A^S$ and $\ker\lambda_{A'}^S$, where $\theta$ is the automorphism of $(S^1)^{2n+1}$ which maps $\h$ to $\h'$, proving the lemma.}
\end{proof}

Now we are ready to prove Proposition~\ref{prop: diffeomorphism class of Type 2}.

\begin{proof}[Proof of Proposition~\ref{prop: diffeomorphism class of Type 2}]
Let $\hat{\varphi}_A$ and $\hat{f}$ be the diffeomorphisms in Lemmas~\ref{lem: varphi A induces diffeomorphism} and~\ref{lem: diffeo ZP cap w nonzeo to ZP cap w' nonzero}, {respectively}. We consider the composition
\[
\hat{\varphi}_{A'}\circ\hat{f}\circ\hat{\varphi}_{A}^{-1}:L\cap\{u\neq0\}\to L\cap\{u\neq0\}.
\]
Setting $s_{i}=\sqrt{1-|v_{i}|^{2}}$, we can see that
\begin{align}\label{eq: self-map of L cap u nozero}
\begin{split}
&(\hat{\varphi}_{A'}\circ \hat{f}\circ \hat{\varphi}_A^{-1})([v_1,\dots,v_n,u])\\
&=(\hat{\varphi}_{A'}\circ \hat{f})([s_1,\dots,s_n,v_1,\dots,v_n,u])\\
&=\hat{\varphi}_{A'}\Big(\big[s_1,\dots,s_{i-1},s_i\frac{u}{|u|},s_{i+1}\Big(\frac{u}{|u|}\Big)^2,\dots,s_{j-1}\Big(\frac{u}{|u|}\Big)^2,s_j\frac{u}{|u|},s_{j+1}, \dots,s_n,\\
&\qquad v_1,\dots,v_{i-1},\bar{v}_i,\dots,\bar{v}_{j-1},v_j,\dots,v_n,u\big]\Big)\\
&=[f_1v_1,\dots,f_{i-1}v_{i-1},f_i\bar{v}_i,\dots,f_{j-1}\bar{v}_{j-1},f_jv_j,\dots,f_nv_n,f_0u],
\end{split}
\end{align}
where each $f_{i}$ is a Laurent monomial of $u/|u|$ with coefficient $1$ by \eqref{eq: the map induces diffeomorphism}. Since the coordinates in~\eqref{eq: self-map of L cap u nozero} are homogeneous {with respect to the $S^1$-action on $\tilde{L}$ defined in \eqref{eq: S1-action on tildeL} and $f_0\in S^1$}, we may assume $f_0=1$. Moreover, the weights of $u$ and $v_{i}$ {with respect to} the $S^{1}$-action are $-2$ and {$1$} respectively from~\eqref{eq: S1-action on tildeL}, {the last term in \eqref{eq: self-map of L cap u nozero} must be equal to}
\begin{align*}
[v_1,\dots,v_{i-1},\Big(\frac{\bar{u}}{|u|}\Big)\bar{v}_i,\dots,\Big(\frac{\bar{u}}{|u|}\Big)\bar{v}_{j-1},v_j,\dots,v_n,u],
\end{align*}
where $\bar{u}/|u|=(u/|u|)^{-1}$.

Note that {$L\cap\{|u|=r\}$} for $0<r<1/\sqrt{2}$ is diffeomorphic to $\R P^{2n-1}$. In fact, if we write $v_{i}=x_{i}+\sqrt{-1}y_{i}$, then a diffeomorphism is given by
\begin{equation*}
\begin{split}
L\cap\{|u|=r\} &\to \R P^{2n-1}\\
[v_1,\dots,v_n,r]& \mapsto [x_1,\dots,x_n,y_1,\dots,y_n].
\end{split}
\end{equation*}
Then $\hat{\varphi_{A'}}\circ\hat{f}\circ\hat{\varphi_{A}}^{-1}$ restricted to $|u|=r$ is a self-diffeomorphism $\psi$ of $\R P^{2n-1}$ which maps $[x_1,\dots,x_n,y_1,\dots,y_n]$ to
\begin{equation*}
[x_1,\dots,x_n,y_1,\dots,y_{i-1},-y_i,\dots,-y_{j-1},y_j,\dots,y_n]\in \R P^{2n-1}.
\end{equation*}
Note that $\psi$ is independent of the choice of $r$.
Since {$j-i$ is even}, {$\psi$} is isotopic to the identity {map on $\R P^{2n-1}$. Therefore $\hat{\varphi_{A'}}\circ\hat{f}\circ\hat{\varphi_{A}}^{-1}$ restricted to $L\cap\{|u|\geq 1/2\sqrt{2}\}$ extends to a self-diffeomorphism of $L$ as the identity map on a neighborhood of $L\cap \{u=0\}$. This means that $\hat{f}$ restricted to $(\ZP/\ker\lambda_A^S)\cap \{|w|\geq 1/2\sqrt{2}\}$} extends to a diffeomorphism $\ZP/\ker\lambda_{A}^{S}\to \ZP/\ker\lambda_{A'}^{S}$.
\end{proof}

\section{Projectivity}\label{sec:Projectivity}
It is known that every toric manifold of complex dimension less than $3$ is projective. However, in general, there are many non-projective toric manifolds.
Oda's $3$-fold, which is of Type 2 {in our terminology, is known as the simplest non-projective toric manifold}. In this section, we will check the projectivity of {a toric manifold associated with a fan over $\vcd$.} It turns out that {our toric manifold is projective unless it is of Type 2} and only one toric manifold is projective {in Type 2 in each dimension $\ge 3$ while the others in Type 2 are non-projective} (Propositions~\ref{prop:proj1} and~\ref{prop:proj2}).

First we discuss the toric manifolds of Type 2. Recall that the matrix $A$ of Type~2 {is of the form}:
\begin{equation*}\label{eq: A of Type 2}
\begin{bmatrix}
1 & 0 & \dots &0& a_1 \\
a_2 & 1 & \dots &0& 0\\
\vdots& \vdots &\ddots &\vdots & \vdots \\
0 & 0 & \dots & 1 & 0\\
0 & 0 & \dots & a_n & 1
\end{bmatrix},
\end{equation*}
where $a_i=\pm 1$ and the number of $a_i$'s equal to $1$ is odd.

\begin{prop} \label{prop:proj1}
Toric manifolds of Type 2 are non-projective if there are at least three $1$'s in $\{a_{1},\ldots,a_{n}\}$.
\end{prop}

\begin{proof}
{Since the matrix $A$ is determined by the ordered set $(a_1,\dots,a_n)$, we express $A$ as $(a_1,\dots,a_n)$.}
By projecting the fan $\Delta$ associated with $A=(a_1,\ldots,a_n)$ onto $\Z^n/(\a_k)$, we obtain the fan $\Delta_k$ associated with $A_k=(a_1,\ldots,-a_ka_{k+1},\ldots,a_n)$.\footnote{This is similar to the argument in the proof of Proposition~\ref{prop: classification of A and b}.} Note that the toric manifold $X(\Delta_k)$ is an invariant subvariety of the toric manifold $X(\Delta)$. Furthermore, if $a_k=-1$, then the ordered set $A_k$ has one less $(-1)$'s than the ordered set $A$. We repeat this process until we get the ordered set $A_\ast$ all of whose elements are equal to $1$. Since $A$ has at least three $1$'s from the hypothesis, the size of $A_\ast$ is at least three. If $A_\ast$ has exactly three $1$'s, then the toric manifold associated with $A_\ast$ is Oda's $3$-fold. Since every invariant subvariety of a projective toric variety is also projective, $X(\Delta)$ is non-projective. If $A_\ast$ has at least five $1$'s, then the new ordered set $A'$ obtained from the projection as above has two less $1$'s compared with $A_\ast$. Hence, by repeating projections suitably, we get the ordered set consisting of three $1$'s. Hence, if the ordered set $A$ has at least three $1$'s, then the corresponding toric manifold is non-projective.
\end{proof}

The following proposition {provides a criterion of} whether a {toric manifold} is projective or not.

\begin{prop}\cite[page 70]{fult93}\label{prop: projective criterion}
A {toric manifold} $X$ of complex dimension $n$ is projective
if and only if there is a continuous piecewise linear function $\psi$ on the support $|\Delta_X|=\R^n$ of the fan $\Delta_X$ associated with $X$ that satisfies the following two conditions:
\begin{enumerate}
\item[(C1)] The restriction of $\psi$ to each $n$-dimensional cone $\sigma$, denoted by $\psi_{\sigma}$, is linear.
\item[(C2)] $\psi_{\sigma}$ satisfies $\psi_{\sigma}(\v_i)>\psi(\v_i)$ for $\v_i\not\in\sigma$,
\end{enumerate}
{where $\v_i$'s are the primitive edge vectors in the fan $\Delta_X$.}
\end{prop}

{We think of $\v_i$'s as column vectors as before.  Since $\Delta_X$ is nonsingular, $Q=[\v_{i}]_{\v_{i}\in\sigma}$ is a unimodular matrix and we denote by $\u_i$ the row vector of $Q^{-1}$ corresponding to $\v_i$, i.e., $\u_{i}\v_{j}=\delta_{ij}$}.
Then $\psi_\sigma$ is of the form
\begin{align}\label{eq: linear function on n-dim cones}
\psi_{\sigma}=\sum_{\v_{i}\in\sigma}\psi(\v_{i})\u_{i}.
\end{align}

\begin{prop} \label{prop:proj2}
{Each toric manifold $X(A,\b)$} is projective except when the toric manifold is of Type 2 such that the number of $a_i$'s equal to $1$ is {more than} one.
\end{prop}

\begin{proof}
It is well-known that every Bott manifold is a projective toric manifold. Since a toric manifold of Type 1 is a blow-up of a Bott manifold, it is also projective. For {the} other types, we will find a {continuous piecewise linear} function $\psi$ {on $\R^n$} satisfying (C1) and (C2) in Proposition~\ref{prop: projective criterion}.

{Remember} that the minimal non-faces of $\vcd$ are
\begin{align*}
\{i,n+1\},~\{i,2n+1\}\text{ for }1\leq i\leq n\text{ and } \{n+1,\ldots,2n\}.
\end{align*}
{We may take $\v_1=-\e_1,\dots,\v_n=-\e_n$, so that $\v_{n+1}=\a_1,\ldots,\v_{2n}=\a_n$ and $\v_{2n+1}=\b$.}

We first {deal with} the toric manifolds of Type 0.
\bigskip

\noindent{\bf Type 0.} Recall that
\begin{equation*}
[\a_{1},\dots,\a_{n}]=
\begin{bmatrix}
1 & 0 & \dots &0& -1 \\
-1 & 1 & \dots &0& 0\\
\vdots& \vdots &\ddots &\vdots & \vdots \\
0 & 0 & \dots & 1 & 0\\
0 & 0 & \dots & -1 & 1
\end{bmatrix},\qquad
\b=
\begin{bmatrix}
b_1\\
b_2\\
\vdots\\
\vdots\\
b_n
\end{bmatrix}
~\text{with}~\sum_{i=1}^nb_i=1.
\end{equation*}
Let $b:=\max\{|b_{1}|,\ldots,|b_{n}|\}$, and let $\psi$ be the {piecewise} linear function defined by
\begin{align} \label{eq:psitype0}
\psi(-\e_{i})=\psi(\a_{i})=-1\text{ for }1\leq i\leq n, \text{ and }\psi(\b)=-\frac{n(n-1)}{2}b.
\end{align}
We shall show that {our} $\psi$ satisfies (C2). We divide the proof into two cases according to whether {the} $n$-dimensional cone $\sigma$ contains $\b$ or not.\\

\noindent \underline{Case 1: $\b\not\in\sigma$.} {In this case, $\sigma$ contains some $-\e_i$ as an edge vector.  Through a cyclic permutation on $[n]$, we may assume that $\sigma$ contains $-\e_n$.}
Then the matrix $Q=[\v_i]_{\v_i\in\sigma}$ is equal to $-E_n$ or it is of the form
\begin{equation}\label{eq: the matrix Q of Type 0 b not in sigma}
\begin{split}
Q=[\a_1,&\dots,\a_{r_1},-\e_{r_1+1},\dots,-\e_{r_1+s_1},\dots,\\
&\a_{m+1},\dots,\a_{m+r_q}, -\e_{m+r_q+1},\dots,-\e_{m+r_q+s_q}],
\end{split}
\end{equation}
where $m=\sum_{i=1}^{q-1}(r_{i}+s_{i})$, $m+r_{q}+s_{q}=n$, {$r_1\ge 0$}, $r_{i}\geq1$ {for $i\ge 2$}, and $s_{i}\geq1$ {for $i\ge 1$}.

If $Q=-E_n$, then $Q^{-1}=-E_n$ and $\psi_\sigma=[1,\ldots,1]$ {by \eqref{eq: linear function on n-dim cones} and \eqref{eq:psitype0}}. Hence, $\psi_\sigma(\a_i)=0>\psi(\a_i)$ for $1\le i\le n$ and $\psi_\sigma(\b)=1>\psi(\b)$ {by \eqref{eq:psitype0}}.

If $Q\neq-E_n$, then $Q^{-1}$ is as follows:
\begin{equation} \label{eq: matrix D_rs}
Q^{-1}=\begin{bmatrix}
D_{r_1,s_1} & & \\
& \ddots & \\
& & D_{r_q,s_q}
\end{bmatrix}
\end{equation}
where
\begin{equation*}
D_{r,s} = \begin{array}{c@{\hspace{-.5cm}}l}
\left[ \begin{array}[c]{ccc|ccc}
1& & & & & \\
\vdots&\ddots& & & & \\
1&\dots&1 & & & \\
\hline
-1&\dots&-1 &-1 & & \\
& & & &\ddots & \\
& & & & & -1 \\
\end{array} \right]
&
\begin{array}[c]{@{}l@{\,}l}
\left. \begin{array}{c} \vphantom{0} \\ \vphantom{\vdots}
\\ \vphantom{0} \end{array} \right\} & \text{$r$} \\
\left. \begin{array}{c} \vphantom{0} \\ \vphantom{\vdots}
\\ \vphantom{0} \end{array} \right\} & \text{$s$} \\
\end{array}
\end{array}.
\end{equation*}
Therefore, {it follows from \eqref{eq: linear function on n-dim cones} and \eqref{eq:psitype0} that} $\psi_{\sigma}$ is given by
\begin{equation} \label{eq: psi_sigma_type0}
\begin{split}
\psi_\sigma=[&-1,\dots,-1]Q^{-1}\\
=[&-(r_1-1),-{(r_1-2)},\dots,-1,0,1,\dots,1,\\
&\qquad \dots\dots\dots\dots\dots\dots\dots\dots\\
&-(r_q-1),-{(r_q-2)},\dots,-1,0,1,\dots,1],
\end{split}
\end{equation}
where $0$'s correspond to $\a_{r_1}$, $\a_{r_1+s_1+r_2}$, $\ldots$, $\a_{m+r_q}$ in this order. Then {the vectors $-\e_i$ and $\a_i$ not contained in $\sigma$ satisfy}
\begin{equation}\label{eq: sigma does not contain ei, b}
\begin{split}
&0\le \psi_\sigma(-\e_i)\le r_j-1 \text{ for some }j,\\
&\psi_\sigma(\a_i)=0\text{ or } r_j \text{ for some }j.
\end{split}
\end{equation}
{This together with \eqref{eq:psitype0} shows that (C2) is satisfied for those $-\e_i$ and $\a_i$ not contained in $\sigma$.}
{It also follows from~\eqref{eq:psitype0} and~\eqref{eq: psi_sigma_type0} that}
\begin{align*}
\psi_\sigma(\b)\ge
-\left(\sum_{k=1}^q\frac{r_k(r_k-1)}{2}+\sum_{k=1}^qs_k\right)b
>-\frac{n(n-1)}{2}b=\psi(\b).
\end{align*}
Hence, our $\psi$ satisfies (C1) and (C2).\\

\noindent \underline{Case 2: $\b\in\sigma$.} In this case none of $-\e_i$'s are contained in $\sigma$ and some $\a_i$  is not contained in $\sigma$. Similarly to Case 1, we may assume that $\a_n$ is not contained in $\sigma$. Then
\[
Q=[\a_1,\dots,\a_{n-1},\b]
=\begin{bmatrix}
1 & & & & b_1\\
-1 & 1 & & & b_2\\
& \ddots & \ddots & & \vdots\\
& & -1 & 1 & b_{n-1}\\
& & & -1 & b_n
\end{bmatrix}
\]
{and hence}
\[
Q^{-1}=\begin{bmatrix}
\sum_{k=2}^nb_k& -b_1 &\dots &-b_1 &-b_1\\
\sum_{k=3}^nb_k &\sum_{k=3}^n b_k & \dots & -b_1-b_2 &-b_1-b_2\\
\vdots & \vdots &\ddots &\vdots &\vdots\\
b_n & b_n & \dots &b_n &-\sum_{k=1}^{n-1}b_k\\
1 & 1& \dots & 1 & 1
\end{bmatrix}.
\]
The $i$th component of $\psi_\sigma=[-1,\dots,-1,-\frac{n(n-1)}{2}b]Q^{-1}$ is equal to
\begin{equation} \label{eq:ithcomp}
\sum_{\ell=1}^{i-1}(\sum_{k=1}^{\ell}b_k)-\sum_{\ell=i+1}^n(\sum_{k=\ell}^nb_k)-\frac{n(n-1)}{2}b.
\end{equation}
{The edge vectors not contained in $\sigma$ are $-\e_{1},\ldots,-\e_n$ and $\a_{n}$, and it follows from \eqref{eq:psitype0} and \eqref{eq:ithcomp} that they satisfy}
\[
\begin{split}
\psi_\sigma(-\e_i)&=-\sum_{\ell=1}^{i-1}(\sum_{k=1}^{\ell}b_k)+\sum_{\ell=i+1}^n(\sum_{k=\ell}^nb_k)+\frac{n(n-1)}{2}b>-1=\psi(-\e_i),\\
\psi_\sigma(\a_n)&={\sum_{\ell=2}^n(\sum_{k=\ell}^nb_k)+\sum_{\ell=1}^{n-1}(\sum_{k=1}^\ell b_k)=}(n-1)\sum_{k=1}^nb_k=n-1>-1=\psi(\a_n).
\end{split}
\]
Hence, our $\psi$ satisfies (C1) and (C2).

\bigskip

\noindent{\bf Type 2 having only one $a_i$ equal to $1$ and Type 3.} Recall that, in Type 3,
\begin{equation}\label{eq: matrix of Type 3}
[\a_{1},\ldots,\a_{n}]=
\begin{bmatrix}
1 & 0 & \dots &0& a \\
-1 & 1 & \dots &0& 0\\
\vdots& \vdots &\ddots &\vdots & \vdots \\
0 & 0 & \dots & 1 & 0\\
0 & 0 & \dots & -1 & 1
\end{bmatrix},
\quad
\b=\frac{1}{\det A}\sum_{i=1}^n\a_i
=
\begin{bmatrix}
1\\
0\\
\vdots\\
\vdots\\
0
\end{bmatrix}.
\end{equation}
If $a=1$ in~\eqref{eq: matrix of Type 3}, then $[\a_1,\ldots,\a_n,\b]$ is of Type 2 such that the number of $a_i$'s equal to~$1$ is exactly one. Hence, it remains to prove that the toric manifold associated with $[\a_1,\ldots,\a_n,\b]$ in~\eqref{eq: matrix of Type 3} {is projective when $a\not=0,-1$}.

Let $\psi$ be the {piecewise} linear function defined by
\[
\begin{split}
&\psi(-\e_i)=-1 \ \text{ for $1\le i\le {n}$}, \qquad \psi(\a_i)=-1\ \text{ for }1\le i\le n-1\\
&\psi(\a_n)=-n|a|,\qquad \psi(\b)=-(n-1).
\end{split}
\]
To show that $\psi$ satisfies (C2), we divide the proof into four cases according to whether $\sigma$ contains the edge vectors $\a_n$ and $\b$.\\

\noindent\underline{Case 1: $\b\notin\sigma$ and $\a_n\not\in\sigma$.} In this case, {$-\e_n\in\sigma$ and} the matrix $Q$ is either $-E_n$ or of the form in~\eqref{eq: the matrix Q of Type 0 b not in sigma}. If $Q=-E_n$, then $Q^{-1}=-E_n$ and $\psi_\sigma=[1,\ldots,1]$, and hence
\[
\begin{split}
&\psi_\sigma(\a_i)=0>\psi(\a_i)\quad (1\leq i<n),\\
&\psi_\sigma(\a_n)=a+1>\psi(\a_n), \text{ and}\\
&\psi_\sigma(\b)=1>-(n-1).
\end{split}
\]
Hence, (C2) is satisfied. If $Q\neq -E_n$, then $\psi_\sigma$ is the same as~\eqref{eq: psi_sigma_type0}. Thus (C2) is satisfied for the edge vectors not contained in $\sigma$ except for $\a_{n}$ and $\b$, see~\eqref{eq: sigma does not contain ei, b}. {As for} $\a_{n}$ and $\b$, since $\psi_\sigma(\a_n)=1-a(r_1-1)$ {and} $r_{1}\leq {n-1}$, we get
\[
\psi_\sigma(\a_n)>-(n-2)|a| {\ge} -n|a|=\psi(\a_n),
\]
and
\[
\psi_\sigma(\b)=-(r_1-1)\ge -(n-2)>-(n-1)=\psi(\b).
\]
Hence, our $\psi$ satisfies (C1) and (C2).\\

\noindent\underline{Case 2: $\b\not\in\sigma$ and $\a_{n}\in\sigma$.} In this case, the matrix $Q$ is obtained from the matrix in~\eqref{eq: the matrix Q of Type 0 b not in sigma} by replacing $-\e_{n}$ with $\a_{n}$, that is,
\begin{equation*}
\begin{split}
Q=[\a_1,&\dots,\a_{r_1},-\e_{r_1+1},\dots,-\e_{r_1+s_1},\dots,\\
&\a_{m+1},\dots,\a_{m+r_q}, -\e_{m+r_q+1},\dots,-\e_{m+r_q+s_q{-1}},\a_n].
\end{split}
\end{equation*}
We {consider}  two subcases depending on whether $-\e_{n-1}\in \sigma$ or not.\\

\underline{Subcase 2-1: $-\e_{n-1}{\in}\sigma$.} In this case, $m+r_q<n-1$ and we have
\[
Q^{-1}=\left[\begin{array}{cccc|c}
D_{r_1,s_1} & & & & \\
& \ddots & & &\bar\a\\
& & {D_{r_{q-1},s_{q-1}}}& &\\
& & & D_{r_q,s_q{{-1}}} & \\
\hline
& & & & 1
\end{array}\right],
\]
where $D_{r,s}$ is the matrix~\eqref{eq: matrix D_rs} and $\bar\a=[\underbrace{-a,\dots,-a}_{r_1},a,0,\ldots,0]^T$. Then
\begin{equation*}
\begin{split}
\psi_\sigma=[&-1,\dots,-1,-n|a|]Q^{-1}\\
=[&-(r_1-1),-{(r_1-2)},\dots,-1,0,1\dots,1,\\
&\qquad \dots\dots\dots\dots\dots\dots\dots\dots\\
&{-(r_{q-1}-1),-(r_{q-1}-2),\dots,-1,0,1\dots,1,}\\
&-(r_q-1),-{(r_q-2)},\dots,-1,0, 1,\dots,1, (r_1-1)a-n|a|].
\end{split}
\end{equation*}
Note that $\psi_\sigma$ is the same as~\eqref{eq: psi_sigma_type0} except the last coordinate. Hence, \eqref{eq: sigma does not contain ei, b} shows that $\psi_\sigma$ satisfies (C2) for $-\e_{i}\not\in\sigma$ {except} for $-\e_n$. The function $\psi_\sigma$ also satisfies (C2) for $\a_{i}\not\in\sigma$ except for $\a_{n-1}$. On the other hand, since $r_{1}\leq n-2$, we can see that
\begin{align*}
\begin{split}
&\psi_\sigma(-\e_n)=n|a|-(r_1-1)a>-1=\psi(-\e_n),\\
&\psi_\sigma(\a_{n-1})=1+n|a|-(r_1-1)a>-1=\psi(\a_{n-1}), \text{ and}\\
&\psi_\sigma(\b)=-(r_1-1)>-(n-1)=\psi(\b).
\end{split}
\end{align*}
This proves that $\psi_\sigma$ satisfies (C2).\\

\underline{Subcase 2-2: $-\e_{n-1}\notin \sigma$} In this case, $m+r_q=n-1$ {and} we have
\[
Q^{-1}=\begin{bmatrix}
\,D_{r_1,s_1} & & & \overline{A}\,\\
& \ddots & & \\
& & D_{r_{q-1},s_{q-1}} & \\
& & & D
\end{bmatrix},
\]
where $\overline{A}$ is the $(r_1+1)\times (r_q+1)$ matrix and $D$ is the square matrix of size $(r_q+1)$ such that
\[
\overline{A}=
\begin{bmatrix}
-a&\dots&-a\\
\vdots&\dots&\vdots\\
-a&\dots&-a\\
a&\dots&a\end{bmatrix}
\text{ and }
D=\begin{bmatrix}
1 & & & \\
1& 1& & \\
\vdots &\vdots & \ddots & \\
1& 1&\dots & 1
\end{bmatrix}.
\]
Then
\begin{align*}
\begin{split}
\psi_\sigma=[&-1,\dots,-1,-n|a|]Q^{-1}\\
=[&-(r_1-1),-{(r_1-2)},\dots,-1,0,1\dots,1,\\
&\qquad \dots\dots\dots\dots\dots\dots\dots\dots\\
&-(r_{q-1}-1),-{(r_{q-1}-2)},\dots,-1,0,1\dots,1,\\
& (r_1-1)a-r_q-n|a|,\dots,(r_1-1)a-n|a|].
\end{split}
\end{align*}
Since $\psi_\sigma$ is the same as~{\eqref{eq: psi_sigma_type0}} except for the coordinates from $m+1$ to $n$, it is enough to check (C2) for $-\e_i\not\in\sigma$ $(m+1\leq i\leq n)$, $\a_{m}$, and $\b$. Note that $n=m+r_q+1$. Since $r_{1}\leq n-2$, we can see that
\begin{equation*}
\begin{array}{ll}
\psi_\sigma(-\e_i)=n|a|+(n-i)-(r_1-1)a>-1=\psi(-\e_i)&(m+1\le i\le n),\\
\psi_\sigma(\a_{m})=1+n|a|+r_q-(r_1-1)a> -1=\psi(\a_{m}),&\\
\psi_\sigma(\b)=-(r_1-1)>-(n-1)=\psi(\b).&
\end{array}
\end{equation*}
Hence, our $\psi$ satisfies (C1) and (C2).\\

\noindent\underline{Case 3: $\b\in\sigma$ and $\a_n\not\in\sigma$.} Since $Q=[\a_1,\dots,\a_{n-1},\b]$, we have
\[
Q^{-1}=\begin{bmatrix}
0&-1 &\dots & -1\\
\vdots &\ddots & \ddots& \vdots \\
0&\dots & 0 &-1 \\
1 &\dots & 1 & 1
\end{bmatrix}.
\]
Hence
\[
\psi_\sigma=[-1,\dots,-1,-(n-1)]Q^{-1}=[-(n-1),{-(n-2)},\dots,-1,0].
\]
The vectors $-\e_{i}$ for $1\leq i\leq n$ and $\a_{n}$ {are not contained in} $\sigma$, and we can see that
\[
\begin{split}
&\psi_\sigma(-\e_i)=n-i>-1=\psi(-\e_i),\\
&\psi_\sigma(\a_n)=-(n-1)a>-n|a|=\psi(\a_n),
\end{split}
\]
since $a\neq0$. Thus, $\psi_\sigma$ satisfies (C2).\\

\noindent\underline{Case 4: $\b\in\sigma$ and $\a_{n}\in\sigma$.} Since the matrix $Q=
[\a_1,\dots,\a_{k-1},\b,\a_{k+1},\dots,\a_n]$ for some $1\le k\le n-1$,
we obtain
\[
Q^{-1}=\begin{bmatrix}
0 & -1 &\dots&-1 & & & \\
\vdots & \ddots &\ddots &\vdots & & & \\
0&\dots & 0 & -1 & & & \\
1 & \dots & 1 & 1&-a&\dots &-a\\
& & & & 1 & & \\
& & & & \vdots &\ddots & \\
& & & & 1&\dots&1
\end{bmatrix},
\]
where the number of $0$'s on the main diagonal is $k-1$.
Hence
\[
\begin{split}
\psi_\sigma=&[-1,\dots,-1,-(n-1),-1,\dots,-1,-n|a|]Q^{-1}\\
=&[-(n-1),\dots,-(n-1)+k-2,-(n-1)+k-1, \\
&\quad (n-1)a-(n-k-1)-n|a|,\dots,(n-1)a-n|a|].
\end{split}
\]
The edge vectors not contained in $\sigma$ are $\a_k$ and $-\e_i$ {for $i=1\dots,n$, and} we can see that
\[
\begin{split}
&\psi_\sigma(-\e_i)
=
\begin{cases}n-i >-1=\psi(-\e_i) &(1\le i\le k),\\
-(n-1)a+(n-i)+n|a|>-1=\psi(-\e_i)&(k+1\le i\le n),
\end{cases}\\
&\psi_\sigma(\a_k)=-(n-1)a-1+n|a|>-1=\psi(\a_k),
\end{split}
\]
since $a\neq 0$. Therefore, $\psi_\sigma$ satisfies (C2).

This completes the proof of the proposition.
\end{proof}

\appendix
\setcounter{section}{-1}
\section{}\label{sec:appendix}
\renewcommand{\thesection}{A}
{In this section, we provide the proof of Proposition \ref{prop: keyproposition} and give some remarks.} Note that, in this section, the term fan does not necessarily mean a rational fan.

First, we regard a moment-angle manifold as a submanifold of $\C^m$ in the following way. Suppose that a simple $n$-polytope $P$ with $m$ facets is realized in $\R^n$ by
$$ P=\{{\x}\in\R^n\ |\ \langle{\n}_i,\x\rangle+ \gamma_i\ge 0\quad(1\le i\le m)\} $$
where ${\n}_i\in \R^n$, $\gamma_i{\in \R}$ and $\langle\ ,\ \rangle$ denotes the standard scalar product on $\R^n$. Then we identify $\ZP$ with the image of
$$ \{(\x,\z)\in P\times\C^m\ |\ |z_i|^2=\langle{\n}_i,\x\rangle+\gamma_i\quad(1\le i\le m)\},\quad \z=(z_1,\ldots,z_m), $$
under the projection to the second component. Note that this set does not depend on whether $\x$ runs over $P$ or $\R^n$. Since we can choose $n$ vectors out of ${\n}_1,\ldots,{\n}_m$ so that they are linearly independent, we can delete $\x$ from the above $m$ equations to obtain $f_i\colon\C^m\to\R$ ($i=1,\ldots,m-n$) such that $\ZP{=\bigcap_{i=1}^{m-n} f_i^{-1}(0)}$.

Let $\cK$ be the simplicial complex on $[m]$ whose geometric realization is $\partial P^*$, and put
$$ U(\cK)=\C^m\setminus\bigcup_{I\not\in \cK}\{\z\in\C^m\ |\ z_i=0\ (i\in I)\}. $$
If we denote by $\lambda^\R\colon (\R_{>0})^m\to(\R_{>0})^n$ the homomorphism {sending $(y_1,\ldots,y_m)$ to $y_1^{\n_1}\cdots y_m^{\n_m}$, where $y^{\u}=(y^{u_1},\dots, y^{u_n})\in (\R_{>0})^n$ for $y\in \R_{>0}$ and $\u=(u_1,\dots,u_n)\in \R^n$}, then we have the following.

\begin{prop}[cf. {\cite[Theorem 3.3]{pa-us12}}]\label{prop: app_1}
The group $\ker\lambda^{\R}$ acts on $U(\cK)$ freely and properly. Moreover, the inclusion $\ZP\to \C^m$ induces an $(S^1)^m$-equivariant diffeomorphism $\ZP\to U(\cK)/\ker\lambda^{\R}$.
\end{prop}

More generally, for any simplicial fan $(\cK,\{{\n}_i\}_{i=1}^m)$, $\ker\lambda^{\R}$ acts on $U(\cK)$ freely and properly. The following lemma is due to Hiroaki Ishida.

\begin{lemm}\label{lemm: app_2}
Let ${\n}_i(t)~(i=1,\ldots,m)$ be a smooth function $[0,1]\to\R^n$ and suppose that $(\cK,\{{\n}_i(t)\}_{i=1}^m)$ is a simplicial fan for each t in $[0,1]$. Then $U(\cK)/\ker\lambda^{\R}(0)$ is $(S^1)^{{m}}$-equivariantly diffeomorphic to $U(\cK)/\ker\lambda^{\R}(1)$.
\end{lemm}

\begin{proof}
Let $Y$ be the quotient space $U(\cK)\times[0,1]/{\sim}$ where the equivalence relation $\sim$ is defined so that $({\z},t)\sim({\z'},t')$ if and only if $t=t'$ and {$\z$, $\z'$} are in the same orbit of $\ker\lambda^\R(t)$. Then we can easily verify that the projection to the second component descends to an $(S^1)^m$-equivariant smooth fiber bundle $Y\to[0,1]$. Since $[0,1]$ is contractible, we obtain the lemma.
\end{proof}

Let $\Delta=(\cK,{\cV})$ be a complete nonsingular fan {(see Section~\ref{sec:Preliminaries})}. Then {the homomorphism $\lambda_{\cV}\colon(\C^*)^m\to(\C^*)^n$ was} defined in the same way as  $\lambda^\R$ {above}.  {We note that $(\C^*)^m=(\R_{>0})^m\times (S^1)^m$ and denote by $\lambda_{\cV}^\R$ and $\lambda_{\cV}^S$ the restrictions of $\lambda_{\cV}$ to $(\R_{>0})^m$ and $(S^1)^m$ respectively.} Then since
$$ X(\Delta)=U(\cK)/\ker\lambda_{{\cV}}=(U({\cK})/\ker\lambda_{{\cV}}^\R)/\ker\lambda_{{\cV}}^S, $$
we have the following corollary {from} Propostion \ref{prop: app_1} and Lemma \ref{lemm: app_2}.

\begin{coro}\label{coro: app_3}
{If there is a smooth deformation of simplicial fan between a complete nonsingular {fan} $\Delta=(\cK,\cV)$ and the normal fan of a simple polytope $P$ which is not necessarily a rational fan, then the toric manifold $X(\Delta)$} associated with $\Delta$ is $(S^1)^n$-equivariantly diffeomorphic to $\ZP/\ker\lambda_{{\cV}}^S$.
\end{coro}

\begin{proof}[{P}roof of Proposition~\ref{prop: keyproposition}]
Consider the deformation given by
\begin{equation*}
A(t)=(\a_1(t),\ldots,\a_n(t))=\begin{bmatrix}
1 & 0 & \dots &0& a_1t \\
a_2t & 1 & \dots &0& 0\\
\vdots& \vdots &\ddots &\vdots & \vdots \\
0 & 0 & \dots & 1 & 0\\
0 & 0 & \dots & a_nt & 1
\end{bmatrix},\qquad\b(t)=\frac{1}{2}\sum_{i=1}^n\a_i(t).
\end{equation*}
We can easily verify that {this is a smooth family of simplicial fans. Since the column vectors {of} $(-E_n,\ A(0),\ \b(0))$  are the normal vectors of $P$ {in~\eqref{eq: polytope},} where $E_n$ denotes the identity matrix}, the proposition {follows from Corollary~\ref{coro: app_3}}.
\end{proof}

{We finish this paper with three remarks.}

\begin{rema}\label{rema:over vcd}
Every toric manifold of Type 2 is over $\vt(I^n)$ from Proposition~\ref{prop: keyproposition}. Hence, together with Proposition~\ref{prop:proj2}, we can conclude that the toric manifold associated with a fan over $\vcd$ is over $\vt(I^n)$.
\end{rema}
\begin{rema}
Delaunay \cite{dela05} proves that if there exists a projective toric manifold over a simple $3$-dimensional polytope $P$, then $P$ has at least one triangular or quadrangular face (see also \cite{ayze16}).  On the other hand, Suyama \cite{suya15} shows that any triangulation of $2$-sphere with at most $18$ vertices can be the underlying simplicial complex of a complete nonsingular fan of dimension~$3$.  Combining these two results, we {can} see that there exist many non-projective toric manifolds of complex dimension $3$.
\end{rema}
\begin{rema}
Hirzebruch surface $F_a$, where $a$ is a nonnegative integer, is the total space of the projective bundle $\CP(\uC\oplus\mathcal{O}(a))$ over $\CP^1$, where $\uC$ denotes the trivial line bundle over $\CP^1$ and $\mathcal{O}(a)$ denotes the complex line bundle over $\CP^1$ whose first Chern class is $a$ times a generator of $H^2(\CP^1)$.  Therefore, $F_a$ is a Bott manifold.  As is well-known, $F_a$'s are not isomorphic to each other as varieties but $F_a$ and $F_b$ are diffeomorphic if and only if $a\equiv b\pmod{2}$.  The latter fact can be proved in an elementary way but it can also be proved using deformations of complex structures as is well-known.  So, looking at Propositions~\ref{prop:diffeomorphism class of Type 0} and \ref{prop: diffeomorphism class of Type 2}, it would be interesting to ask whether there exist deformations of complex structures on our toric manifolds of Type 0 or 2. {See~\cite{sa03} for related work.}
\end{rema}

\section*{Acknowledgements}  The authors thank Suyoung Choi for providing the proof of Lemma~\ref{lem: cohomology ring of Type 2} when $n=3$ and Hiroaki Ishida for his help on moment-angle manifolds.

\end{document}